\documentclass[leqno,a4paper]{article}
\usepackage{amssymb,amsmath,amsthm,verbatim}

\title{\textbf{Invariants for the Lagrangian Equivalence Problem}}

\author{\textsc{M. Castrill\'on L\'opez}
\and \textsc{J. Mu\~{n}oz Masqu\'e}
\and \textsc{E. Rosado Mar\'{\i}a}}

\date{}

\newtheorem{theorem}{Theorem}[section]
\newtheorem{proposition}[theorem]{Proposition}
\newtheorem{lemma}[theorem]{Lemma}

\theoremstyle{remark}

\newtheorem{example}[theorem]{Example}

\begin{document}

\maketitle

\begin{abstract}
\noindent Let $M$ be a connected smooth manifold, let $\operatorname{Aut}(p)$
be the group automorphisms of the bundle $p\colon \mathbb{R}\times M\to \mathbb{R}$,
and let $q\colon J^1(\mathbb{R},M)\times \mathbb{R\to }J^1(\mathbb{R},M)$
be the canonical projection.
Invariant functions on $J^r(q)$ under the natural action of $\operatorname{Aut}(p)$
are discussed in relationship with the Lagrangian equivalence problem.
The second-order invariants are determined geometrically
as well as some other higher-order invariants for $\dim M\geq 2$.
\end{abstract}

\medskip


\smallskip


\smallskip


\smallskip

\noindent\emph{Acknowledgements:\/} M.C.L. was partially supported by MINECO (Spain)
under grant MTM2015--63612-P.

\section{Introduction and preliminaries}
\subsection{Statement of the problem}
Let $M$ be a connected smooth manifold of dimension $m$. An automorphism
of the projection $p\colon \mathbb{R}\times M\to \mathbb{R}$, $p(x,y)=x$,
is a pair $\phi \in \operatorname{Diff}\mathbb{R}$,
$\Phi \in \operatorname{Diff}(\mathbb{R}\times M)$ making the following
diagram commutative:
\[
\begin{array}
[c]{ccc}
\mathbb{R}\times M & \overset{\Phi  }{\longrightarrow} & \mathbb{R}\times M\\
\downarrow\scriptstyle{p} &  & \downarrow \scriptstyle{p}\\
\mathbb{R} & \overset{\phi }{\longrightarrow} & \mathbb{R}
\end{array}
\]
Let $\operatorname{Aut}(p)$ be the group of automorphisms of $p$.
The diffeomorphism $\phi $ is completely determined by $\Phi $;
hence, $\operatorname{Aut}(p)$ is a subgroup in
$\operatorname{Diff}(\mathbb{R}\times M)$.
We denote by $J^{r}(\mathbb{R},M)$ the bundle of jets of order $r$
of (local) sections of the submersion $p$ and by
$\Phi ^{(r)}\colon J^r(\mathbb{R},M)\to  J^{r}(\mathbb{R},M)$
the $r$-th jet prolongation of $\Phi $, i.e.,
\[
\Phi ^{(r)}(j_x^r\sigma )
=j_{\phi (x)}^r\left( \Phi  \circ\sigma \circ \phi ^{-1}\right)  ,
\quad\forall  j_x^r\sigma \in J^r(\mathbb{R},M).
\]

Let $\mathcal{L},\bar{\mathcal{L}}\colon J^1(\mathbb{R},M)
\simeq \mathbb{R}\times TM\to \mathbb{R}$ be two first-order
(time dependent) Lagrangians on $M$. The Lagrangians densities
$\mathcal{L}dx$, $\mathcal{\bar{L}}dx$ are said to be equivalent
if there exists $\Phi \in \operatorname{Aut}(p)$ such that,
$(\Phi  ^{(1)})^\ast \left( \mathcal{L}dx\right)  =\mathcal{\bar{L}}dx$,
i.e., $\mathcal{\bar{L}}dx=(\mathcal{L\circ}\Phi ^{(1)})\phi ^\ast (dx)$,
or even,
\begin{equation}
\mathcal{\bar{L}}=\frac{d\phi }{dx}
\left( \mathcal{L\circ }\Phi ^{(1)}\right) .
\label{rho}
\end{equation}

The equivalence problem is one of the basic questions in the Calculus of Variations
and has been dealt with in several works. The notion of equivalence itself
has different interpretations. Equivalence up to an automorphism of $p$
(also known as \emph{fiber preserving} equivalence) is probably the most natural one
and is the approach we follow in this article, although there are other possibilities,
such as equivalence under the group $\mathrm{Diff}(\mathbb{R}\times M)$
(point transformations) or even larger groups. The usual tool used to solve
the equivalence problem in the literature has been the Cartan method.
This procedure is quite algorithmic but the computations become early too complex.
However, partial results have been obtained: Equivalence of quadratic Lagrangians
in \cite{Olver1} or \cite{Bad} (with respect to the Euler-Lagrange equations
in the last case), equivalence of Lagrangian for $M=\mathbb{R}$ in \cite{KO}
and \cite{Olver2}, or the equivalence problems for field theory Lagrangians defined
in the plane $\mathbb{R}^2$ and $M= \mathbb{R}$ in \cite{GS}. Our results are obtained
for an arbitrary manifold $M$, agreeing with the previous results in the particular case
$M= \mathbb{R}$. For that, we follow a different method connected with the notion
of invariant and infinitesimal transformations, defined as follows.

We first note that sections of the projection
\[
\begin{array}
[c]{l}
q\colon J^1(\mathbb{R},M)\times \mathbb{R}\to J^1(\mathbb{R},M),\\
q(j_x^1\sigma ,\lambda)=j_x^1\sigma ,
\end{array}
\]
correspond bijectively with functions in $C^\infty (J^1(\mathbb{R},M))$,
i.e., with first-order Lagrangians on the fibred manifold
$p\colon \mathbb{R}\times M\to \mathbb{R}$. The section
$s_{\mathcal{L}}\colon J^1(\mathbb{R},M)
\to J^1(\mathbb{R},M)\times \mathbb{R}$ of $q$ corresponding to
$\mathcal{L}\in C^\infty (J^1(\mathbb{R},M))$ is given by
$s_{\mathcal{L}}(j_x^1\sigma )=(j_x^1\sigma ,\mathcal{L}(j_x^1\sigma ))$.
In what follows $s_{\mathcal{L}}$ and $\mathcal{L}$ will be identified.

Let $\tilde{\Phi }^{(1)}\colon J^1(\mathbb{R},M)
\times \mathbb{R}\to J^1(\mathbb{R},M)\times \mathbb{R}$
be the automorphism of $q$ defined as follows:
\begin{equation}
\tilde{\Phi  }^{(1)}(j_x^1\sigma ,\lambda )
=(\Phi  ^{(1)}(j_x^1\sigma ),\alpha^{-1}\lambda ),
\label{tilde_Phi}
\end{equation}
where $\alpha \in \mathbb{R}^\ast $ is determined by,
$(\Phi ^\ast dx)_x=\alpha (dx)_x$. If $(\Psi ,\psi )$
is another automorphism of $p$ and $(\psi ^\ast dx)_x=\beta (dx)_x$,
then
\[
((\Phi \circ \psi )^\ast dx\mathbf{)}_x=\psi ^\ast (\Phi ^\ast dx)_x
=\psi ^\ast (\alpha dx)_x=\alpha\beta(dx)_x.
\]
Hence
\begin{align*}
\tilde{\Phi }^{(1)}
\Bigl(
\tilde{\Psi}^{(1)}(j_x^1\sigma ,\lambda )
\Bigr)
& =\tilde{\Phi }^{(1)}(\Psi^{(1)}(j_x^1\sigma ),\beta^{-1}\lambda )\\
& =
\Bigl(
\Phi ^{(1)}\left( \Psi^{(1)}(j_x^1\sigma )\right) ,
\alpha ^{-1}\beta ^{-1}\lambda
\Bigr)
\\
& =
\Bigl(
\left( \Phi \circ \Psi\right) ^{(1)}(j_x^1\sigma ),
\left( \alpha \beta\right) ^{-1}\lambda
\Bigr) .
\end{align*}
In other words,
$(\widetilde{\left(  \Phi\circ\Psi\right)  })^{(1)}=\tilde{\Phi}^{(1)}
\circ\tilde{\Psi}^{(1)}$.
\subsection{Differential invariants}
\begin{proposition}
\label{proposition1}With the preceding notations,
$\left( (\Phi ^{-1})^{(1)}\right) ^\ast (\mathcal{L}dx)
=\mathcal{\bar{L}}dx$ for $\Phi \in \mathrm{Aut}(p)$ if and only if,
$\tilde{\Phi }^{(1)}\circ s_{\mathcal{L}}\circ(\Phi ^{(1)})^{-1}
=s_{\mathcal{\bar{L}}}$. Accordingly, if the Lagrangian
densities $\mathcal{L}dx$ and $\mathcal{\bar{L}}dx$ are equivalent,
then every smooth function $I\colon J^r(q)\to \mathbb{R}$ which is invariant
under the subgroup of $\operatorname{Aut}(q)$ of automorphisms of the form
$(\tilde{\Phi }^{(1)},\Phi ^{(1)})$, $\Phi \in \operatorname{Aut}(p)$, takes
the same value on the two $r$-jets of the sections $s_{\mathcal{L}}$ and
$s_{\mathcal{\bar{L}}} $; namely,
\[
I\left( j_{j_x^1\sigma }^rs_{\mathcal{L}}\right)
=I\left( j_{j_x^1\sigma }^rs_{\mathcal{\bar{L}}}\right) ,
\quad\forall  j_x^1\sigma \in J^1(\mathbb{R},M).
\]
Therefore, invariant functions provide a method to distinguish non-equivalent
Lagrangian densities.
\end{proposition}
\begin{proof}
According to \eqref{rho}, the equation $((\Phi ^{-1})^{(1)})^\ast
(\mathcal{L}dx)=\mathcal{\bar{L}}dx$ is equivalent to say, $\mathcal{\bar{L}}
=\Phi ^{\prime-1}(\mathcal{L\circ(}\Phi ^{-1})^{(1)})$. We have
\begin{align*}
\left( \tilde{\Phi }^{(1)}\circ s_{\mathcal{L}}\circ (\Phi ^{(1)})^{-1}\right)
\left( j_x^1\sigma \right) & =\left( \tilde{\Phi }^{(1)}\circ
s_{\mathcal{L}}\right) \left( (\Phi ^{(1)})^{-1}\left( j_x^1
\sigma \right) \right) \\
& =\tilde{\Phi }^{(1)}\left( (\Phi ^{(1)})^{-1}\left( j_x^1\sigma \right)
,\mathcal{L}\left( \left( \Phi ^{-1}\right) ^{(1)}\left( j_x^1
\sigma \right) \right) \right) \\
& =\left( j_x^1\sigma ,\Phi ^{\prime-1}\left( \mathcal{L\circ}\left(
\Phi ^{-1}\right) ^{(1)}\right) \left( j_x^1\sigma \right) \right) \\
& =s_{\mathcal{\bar{L}}}\left( j_x^1\sigma \right) .
\end{align*}
\end{proof}

The approach we follow for the determination of the invariant functions
will be through the infinitesimal representation of projectable (resp.\ vertical)
vector fields $X$ of $\mathbb{R}\times M\to \mathbb{R}$, that is, we will study
the functions $I:J^r(q)\to \mathbb{R}$ such that $(\tilde{X}^{(1)})^{(r)}(I)=0$,
for $r\leq 2$, according to the notation of the next section. In principle,
these functions will only be invariant under the connected component
of the identity $\mathrm{Aut}(p)_0$ (resp.\ $\mathrm{Aut}^v(p)_0$
where $\mathrm{Aut}^v(p)$ is the group of vertical automorphisms of $p$),
so that the number of independent invariant functions under the full group
$\mathrm{Aut}(p)$ (resp.\ $\mathrm{Aut}^v(p)$) is less or equal to those obtained
by that infinitesimal representation. However, we will conclude the main result
by providing explicitly enough functions that are invariant under $\mathrm{Aut}(p)$
(resp.\ $\mathrm{Aut}^v(p)$).
\subsection{Zero-order invariants}
Let
\begin{equation}
X=u\tfrac{\partial }{\partial x}+v^\alpha \tfrac{\partial }{\partial y^\alpha },
\quad u\in C^\infty (\mathbb{R}),\; v^\alpha \in C^\infty (\mathbb{R}\times M),
\label{X}
\end{equation}
be the local expression of a $p$-projectable vector field on a fibred
coordinate system $(x,y^\alpha )$, $y^\alpha \in C^\infty (\mathbb{R}\times M)$,
$1\leq\alpha\leq m$, for the natural projection $p\colon \mathbb{R}\times M
\to \mathbb{R}$, where we have used the Einstein summation convention.
We have (cf.\ \cite[2.1.1]{MR})
\begin{equation}
\left\{
\begin{array}
[c]{l}
X^{(1)}=u\tfrac{\partial }{\partial x}
+v^\alpha \tfrac{\partial }{\partial y^\alpha }
+v_{1}^\alpha \tfrac{\partial }{\partial\dot{y}^\alpha },
\smallskip\\
v_1^\alpha =\tfrac{\partial v^\alpha }{\partial x}
-\tfrac{du}{dx}\dot{y}^\alpha
+\tfrac{\partial v^\alpha }{\partial y^\beta }\dot{y}^\beta ,
\end{array}
\right.  \label{X(1)}
\end{equation}
where $(x,y^\alpha ,\dot{y}^\alpha )$ is the coordinate system on
$J^1(\mathbb{R},M)$ induced from $(x,y^\alpha )$, $1\leq \alpha \leq m$;
i.e., $\dot{y}^\alpha (j_{x_0}^1\sigma )=
(d(y^\alpha \circ \sigma )/dx)(x_0)$.

Let us compute $\tilde{X}^{(1)}$. Let
$z\colon J^1(\mathbb{R},M)\times \mathbb{R}\to \mathbb{R}$
be the projection onto the second factor, i.e.,
$z(j_x^1\sigma ,\lambda )=\lambda $. The functions
$(x,y^\alpha ,\dot{y}^\alpha ,z)$ are a system of coordinates on
$J^1(\mathbb{R},M)\times \mathbb{R}$. If $X$ is the infinitesimal
generator of a pair of one-parameter groups
$\phi _t\in \operatorname{Diff}\mathbb{R}$,
$\Phi _t\in \operatorname{Diff}(\mathbb{R}\times M)$, then according
to \eqref{tilde_Phi} we have
$z\circ\tilde{\Phi }^{(1)}=\left. z\right/ \tfrac{d\phi _t}{dx}$.
Hence
\begin{align*}
\tilde{X}^{(1)}\left( x\right)
& =\left.
\tfrac{\partial }{\partial t}\right\vert _{t=0}(x\circ\phi _t^{(1)})
=\left.
\tfrac{\partial }{\partial t}\right\vert _{t=0}\left( x\circ\Phi _{t}\right)
=u,\\
\tilde{X}^{(1)}\left( y^\alpha \right)
& =\left.
\tfrac{\partial }{\partial t}\right\vert _{t=0}(y^\alpha \circ\Phi _{t}^{(1)})
=\left.
\tfrac{\partial }{\partial t}\right\vert _{t=0}
\left( y^\alpha \circ \Phi _{t}\right) =v^\alpha ,\\
\tilde{X}^{(1)}\left( \dot{y}^\alpha \right)
& =\left.
\tfrac{\partial }{\partial t}\right\vert _{t=0}
(\dot{y}^\alpha \circ\Phi _{t}^{(1)})=v_1^\alpha ,
\end{align*}
\begin{align*}
\tilde{X}^{(1)}(z)
& =z\left.  \tfrac{\partial }{\partial t}\right\vert _{t=0}
\left( \tfrac{d\Phi _{t}}{dx}\right) ^{-1}\\
& =-z\left.
\tfrac{\partial }{\partial t}\right\vert _{t=0}
\left(
\tfrac{d\Phi _{t}}{dx}\right) \\
& =-\tfrac{du}{dx}z.
\end{align*}
Therefore $\tilde{X}^{(1)}=u\tfrac{\partial }{\partial x}
+v^\alpha \tfrac{\partial }{\partial y^\alpha }
+v_1^\alpha \tfrac{\partial }{\partial\dot{y}^\alpha }
-\tfrac{du}{dx}z\tfrac{\partial }{\partial z}$.

Let $\mathcal{D}^{0}$ be the involutive distribution on
$J^1(\mathbb{R},M)\times\mathbb{R}$ generated by the vector fields
$\tilde{X}^{(1)}$, $X$ being an arbitrary $p$-projectable
vector field. The first integrals of $\mathcal{D}^{0}$
are the zero-order invariant functions. Taking the formulas
\eqref{X(1)} into account, we have
\[
v_1^\alpha (j_x^1\sigma )=\tfrac{\partial v^\alpha }{\partial x}
(\sigma (x))+\tfrac{\partial(y^\beta \circ\sigma )}{\partial x}(x)
\tfrac {\partial v^\alpha }{\partial y^\beta }(\sigma (x))
-\tfrac{\partial (y^\alpha \circ\sigma )}{\partial x}(x)
\tfrac{du}{dx}(x),
\]
and evaluating $\tilde{X}^{(1)}$ at a point $(j_x^1\sigma ,\lambda )$,
\begin{align*}
\left( \! \tilde{X}^{(1)}\!\right) _{(j_x^1\sigma ,\lambda )}\! &
=\! u(x)\left(
\! \tfrac{\partial }{\partial x}
\! \right) _{(j_x^1\sigma ,\lambda)}
\! +\!v^\alpha (\sigma (x))
\left(
\! \tfrac{\partial }{\partial y^\alpha }
\! \right) _{(j_x^1\sigma ,\lambda)}\\
& \! +\! \tfrac{\partial v^\alpha }{\partial x}(\sigma (x))
\left(
\! \tfrac{\partial }{\partial \dot{y}^\alpha }\!
\right) _{(j_x^1\sigma ,\lambda) }
\! +\! \tfrac{\partial v^\alpha }{\partial y^\beta }
(\sigma (x))\tfrac{\partial(y^\beta \circ\sigma )}{\partial x}(x)
\left(
\! \tfrac{\partial }{\partial \dot{y}^\alpha }
\!\right) _{(j_x^1
\sigma ,\lambda )}\\
&  \! -\!\lambda
\tfrac{du}{dx}(x)\left(
\!\tfrac{\partial }{\partial z}
\!\right) _{(j_x^1\sigma ,\lambda )}
\! -\! \tfrac{du}{dx}(x)
\tfrac{\partial(y^\alpha \circ \sigma )}{\partial x}(x)
\left( \!\tfrac{\partial }{\partial\dot{y}^\alpha }
\!\right) _{(j_x^1\sigma ,\lambda)}
\end{align*}
\begin{align*}
& =u(x)\left(
\tfrac{\partial }{\partial x}
\right) _{(j_x^1\sigma ,\lambda )}
+v^\alpha (\sigma (x))\left(
\tfrac{\partial }{\partial y^\alpha }
\right) _{(j_x^1\sigma ,\lambda )}
+\tfrac{\partial v^\alpha }{\partial x}(\sigma (x))
\left(
\tfrac{\partial }{\partial\dot{y}^\alpha }
\right) _{(j_x^1\sigma ,\lambda )}\\
& +\tfrac{\partial v^\alpha }{\partial y^\beta }(\sigma (x))
\left(
\dot{y}^\beta \tfrac{\partial }{\partial \dot{y}^\alpha }
\right) _{(j_x^1\sigma ,\lambda )}
-\tfrac{du}{dx}(x)\left( z\tfrac{\partial }{\partial z}
+\dot{y}^\alpha \tfrac{\partial }{\partial \dot{y}^\alpha }
\right) _{(j_x^1\sigma ,\lambda )}.
\end{align*}
As the values $u(x)$, $(du/dx)(x)$, $v^\alpha (\sigma (x))$,
$(\partial v^\alpha /\partial x)(\sigma (x))$,
$(\partial v^\alpha /\partial y^\beta )(\sigma (x))$ are arbitrary,
it follows that the vector fields $\partial /\partial x$,
$\partial/\partial y^\alpha $, $\partial/\partial\dot{y}^\alpha $,
$z\partial/\partial z$ are a local basis of the distribution
$\mathcal{D}^{0}$ on the dense open subset $z\neq 0$; the generic rank
of $\mathcal{D}^{0}$ thus coincides with the dimension
of $J^1(\mathbb{R} ,M)\times\mathbb{R}$ ; hence the only zero-order
invariants are the constants.
\section{Jet-prolongation formulas\label{prolongation1}}
Let $q\colon N\times\mathbb{R\to }N$ be the canonical projection onto
the first factor, where $N$ is smooth manifold of dimension $\nu $
and local coordinates $(t^1,\dotsc,t^\nu )$.

Let $z\colon N\times \mathbb{R}\to \mathbb{R}$ be the projection
onto the second factor. The lift by infinitesimal contact transformation
to the jet bundle $J^r(q)$ of a $q$-projectable vector field on
$N\times \mathbb{R}$ with local expression
\[
Y=\xi ^a\left( t^1,\dotsc,t^\nu \right) \tfrac{\partial }{\partial t^a
}+\eta \left( t^1,\dotsc,t^\nu ,z\right) \tfrac{\partial }{\partial z},
\]
is given by,
\begin{equation}
Y^{(r)}=\xi ^a\left( t^1,\dotsc,t^\nu \right)
\tfrac{\partial }{\partial t^a}
+\eta \left( t^1,\dotsc,t^\nu ,z\right)
\tfrac{\partial }{\partial z}
+\sum\nolimits_{1\leq|I|\leq r}\eta _I
\tfrac{\partial }{\partial z_I},
\label{Yr}
\end{equation}
where $I=(i_1,\dotsc,i_\nu )\in \mathbb{N}^\nu $ is a multi-index of order
$|I|=i_1+\ldots+i_\nu $, $(t^a,z_I)$, $1\leq a\leq \nu $, $|I|\leq r$, is
the coordinate system induced on $J^r(q)$, namely, $z_0=z$,
$z_{I}(j_\zeta ^r\mathcal{L})
=(\partial^{|I|}\mathcal{L}/\partial t^I)(\zeta )$,
$\mathcal{L}\in C^\infty (N)$, and the function $\eta _I$ is defined as follows:
\begin{equation}
\eta _I=\mathbb{D}^I\left( \eta -\xi ^az_a\right) +\xi ^az_{I+(a)},
\label{eta_I}
\end{equation}
with $(a)=(0,\dotsc,\overset{(a}{1},\dotsc,0)$, $1\leq a\leq \nu $,
$\mathbb{D}_{t^a}=\tfrac{\partial }{\partial t^a}
+\sum\nolimits_{|I|=0}^\infty z_{I+(a)}\tfrac{\partial }{\partial z_I}$
denotes the total derivative with respect to the coordinate $t^a$
and the operator $\mathbb{D}^I$ is given by,
\[
\mathbb{D}^I=\left( \mathbb{D}_{t^1}\right) ^{i_1}\circ \cdots
\circ\left( \mathbb{D}_{t^\nu }\right) ^{i_\nu }.
\]
By using Leibnitz's formula for the $r$-th derivative of a product,
the formula \eqref{eta_I} transforms into the following:
\begin{equation}
\eta _{I}=\mathbb{D}^{I}(\eta )-\sum\nolimits_{J<I}\tbinom{I}{J}
\frac{\partial^{|I-J|}\xi^a}{\partial t^{I-J}}z_{J+(a)}.
\label{eta_I_bis}
\end{equation}

In the case we are dealing with, $N=J^1(\mathbb{R},M)$, $\nu =1+2m$,
$Y=\tilde{X}^{(1)}$, and
\[
\left.
\begin{array}
[c]{l}
(t^a)_{a=1}^{\nu}=(x,y^\alpha ,\dot{y}^\alpha ),\\
\left( \left( \xi^a\right) _{a=1}^\nu {;}\eta\right)
=\left( u,v^\alpha ,v_1^\alpha {;}-\tfrac{du}{dx}z\right) ,
\end{array}
\right\}  1\leq\alpha\leq m,
\]
$u\in C^\infty (\mathbb{R})$, $v^\alpha \in C^\infty (\mathbb{R}\times M)$,
and $v_1^\alpha $ is given in \eqref{X(1)}. Next, instead of the general
formulas above, we use the following ones for $J^1(q)$ in our particular
case:
$(x,y^\alpha ,\dot{y}^\alpha ,z,z_x,z_{y^\alpha },z_{\dot{y}^\alpha })$,
$1\leq\alpha\leq m$.
\section{First-and second-order invariants\label{first/second/order}}
\subsection{First order\label{first/order}}
First of all, let us compute $(\tilde{X}^{(1)})^{(1)}$. By applying \eqref{Yr}
to $Y=\tilde{X}^{(1)}$ for $r=1$, we have
\[
(\tilde{X}^{(1)})^{(1)}=u\tfrac{\partial }{\partial x}
+v^\alpha \tfrac{\partial }{\partial y^\alpha }
+v_1^\alpha \tfrac{\partial }{\partial\dot{y}^\alpha }
-\tfrac{du}{dx}z\tfrac{\partial }{\partial z}
+A\tfrac{\partial }{\partial z_x}
+B_{\alpha}\tfrac{\partial }{\partial z_{y^\alpha }}
+C_{\alpha}\tfrac{\partial }{\partial z_{\dot{y}^\alpha }},
\]
\begin{align}
A\!\! & =\!\! -\tfrac{d^2u}{dx^2}z\!-\! 2\tfrac{du}{dx}z_x
\! -\! \tfrac{\partial v^\alpha }{\partial x}z_{y^\alpha }
\! -\! \tfrac{\partial^2v^\alpha }{\partial x^2}z_{\dot{y}^\alpha }
\! +\! \tfrac{d^2u}{dx^2}\dot{y}^\alpha z_{\dot{y}^\alpha }
\! -\! \tfrac{\partial ^2v^\alpha }{\partial x\partial y^\beta }
\dot{y}^\beta z_{\dot{y}^\alpha },
\label{A}\\
B_\alpha \!\! & =\!\! -\tfrac{du}{dx}z_{y^\alpha }
\! -\! \tfrac{\partial v^\beta }{\partial y^\alpha }z_{y^\beta }
\! -\! \left(
\! \tfrac{\partial ^2v^\beta }{\partial x\partial y^\alpha }
\! +\! \tfrac{\partial^2v^\beta }{\partial y^\alpha \partial y^\gamma }
\dot{y}^\gamma\! \right)
z_{\dot{y}^\beta },
\label{B}\\
C_\alpha\!\! & =\!\! -\tfrac{\partial v^\beta }{\partial y^\alpha }
z_{\dot{y}^\beta }.
\label{C}
\end{align}
Hence
\[
\begin{array}
[c]{rl}
(\tilde{X}^{(1)})^{(1)}= & u\tfrac{\partial }{\partial x}
+v^\alpha \tfrac{\partial }{\partial y^\alpha }\\
& -\tfrac{du}{dx}\left( z\tfrac{\partial }{\partial z}
+\dot{y}^\alpha \tfrac{\partial }{\partial\dot{y}^\alpha }
+2z_x\tfrac{\partial }{\partial z_x}
+z_{y^\alpha }\tfrac{\partial }{\partial z_{y^\alpha }}\right) \\
& +\tfrac{\partial v^\alpha }{\partial x}
\left( \tfrac{\partial }{\partial\dot{y}^\alpha }
-z_{y^\alpha }\tfrac{\partial }{\partial z_x}\right)
+\tfrac{\partial v^\alpha }{\partial y^\beta }
\left(
\dot{y}^\beta \tfrac{\partial }{\partial\dot{y}^\alpha }
-z_{y^\alpha }\tfrac{\partial }{\partial z_{y^\beta }}
-z_{\dot{y}^\alpha }\tfrac{\partial }{\partial z_{\dot{y}^\beta }}
\right) \\
& +\tfrac{d^2u}{dx^2}
\left( -z+\dot{y}^\gamma z_{\dot{y}^\gamma }\right)
\tfrac{\partial }{\partial z_x}
-\tfrac{\partial^2v^\alpha }{\partial x^2}z_{\dot{y}^\alpha }
\tfrac{\partial }{\partial z_x}\\
& -\tfrac{\partial^2v^\beta }{\partial x\partial y^\alpha }
\left( \dot{y}^\alpha z_{\dot{y}^\beta }
\tfrac{\partial }{\partial z_x}+z_{\dot{y}^\beta }
\tfrac{\partial }{\partial z_{y^\alpha }}\right) \\
& -\sum_{\alpha\leq\beta}
\tfrac{\partial^2v^\gamma }{\partial y^{\alpha }\partial y^\beta }
\tfrac{1}{1+\delta_{\alpha\beta}}z_{\dot{y}^\gamma}
\left( \dot{y}^\beta \tfrac{\partial }{\partial z_{y^\alpha }}
+\dot{y}^\alpha \tfrac{\partial }{\partial z_{y^\beta }}\right) ,
\end{array}
\]
and the distribution $\mathcal{D}^1$ generated by all the vector fields
$(\tilde{X}^{(1)})^{(1)}$ on $J^1(q)$ is spanned by
$\tfrac{\partial }{\partial x}$, $\tfrac{\partial }{\partial y^\alpha }$,
and the following vector fields:
\begin{align*}
\chi & =z\tfrac{\partial }{\partial z}+\dot{y}^\alpha
\tfrac{\partial }{\partial\dot{y}^\alpha }
+2z_x\tfrac{\partial }{\partial z_x}
+z_{y^\alpha }\tfrac{\partial }{\partial z_{y^\alpha }},\\
\chi_{\alpha}
& =\tfrac{\partial }{\partial\dot{y}^\alpha }
-z_{y^\alpha }\tfrac{\partial }{\partial z_x},
\quad
1\leq\alpha\leq m,\\
\chi_\alpha ^\beta
& =\dot{y}^\beta \tfrac{\partial }{\partial\dot{y}^\alpha }
-z_{y^\alpha }\tfrac{\partial }{\partial z_{y^\beta }}
-z_{\dot{y}^\alpha }\tfrac{\partial }{\partial z_{\dot{y}^\beta }},
\quad
\alpha ,\beta =1,\dotsc,m,\\
\chi ^\prime
& =\left( -z+\dot{y}^\gamma z_{\dot{y}^\gamma }\right)
\tfrac{\partial }{\partial z_x},\\
\bar{\chi }_\alpha
& =z_{\dot{y}^\alpha }\tfrac{\partial }{\partial z_x},
\quad
1\leq\alpha\leq m,\\
\bar{\chi }_\beta ^\alpha
& =z_{\dot{y}^\beta }\left( \dot{y}^\alpha
\tfrac{\partial }{\partial z_x}
+\tfrac{\partial }{\partial z_{y^\alpha }}\right) ,
\quad
\alpha ,\beta =1,\dotsc,m,\\
\chi _{\alpha \beta}^\gamma
& =\tfrac{1}{1+\delta _{\alpha \beta }}z_{\dot{y}^\gamma }
\left( \dot{y}^\beta \tfrac{\partial }{\partial z_{y^\alpha }}
+\dot{y}^\alpha \tfrac{\partial }{\partial z_{y^\beta }}\right) ,
\quad
1\leq \alpha \leq \beta \leq m, 1\leq \gamma \leq m.
\end{align*}
Let us fix two indices $\alpha _0,\beta _0$. From the formulas above,
on the dense open subset
$U(\alpha _0,\beta _0)
=\{ z\neq 0,\dot{y}^{\beta _0}\neq 0,z_{\dot{y}^{\alpha_0}}
\neq 0,z_{\dot{y}^{\beta_0}}\neq 0\}
\subset J^1(q)$, we have
\begin{align*}
\tfrac{\partial }{\partial z_x}
\! & =\! \tfrac{1}{z_{\dot{y}^{\alpha _0}}}
\bar{\chi }_{\alpha _0},\\
\tfrac{\partial }{\partial\dot{y}^\alpha }
\! & =\!\tfrac{1}{z_{\dot{y}^{\alpha _0}}}
\left( z_{\dot{y}^{\alpha _0}}\chi _\alpha
+z_{y^\alpha }\bar{\chi }_{\alpha _0}\right) ,\\
\tfrac{\partial }{\partial z_{y^\alpha }}
\! & =\! \tfrac{1}{z_{\dot{y}^{\beta _0}}}
\left( \bar{\chi}_{\beta_0}^\alpha -\dot{y}^\alpha
\bar{\chi}_{\beta_0}\right) ,\\
\tfrac{\partial }{\partial z_{\dot{y}^\beta }}
\! & =\! \tfrac{1}{z_{\dot{y}^{\alpha _0}}
z_{\dot{y}^{\beta_0}}}
\left[ \dot{y}^\beta
\left( z_{\dot{y}^{\beta _0}}\chi _{\alpha _0}
+z_{y^{\alpha_0}}\bar{\chi }_{\beta _0}\right)
-z_{y^{\alpha _0}}\bar{\chi }_{\beta _0}^\beta
-z_{\dot{y}^{\beta _0}}\chi _{\alpha _0}^\beta
+z_{y^{\alpha _0}}\dot{y}^\beta \bar{\chi }_{\beta _0}
\right] ,\\
\tfrac{\partial }{\partial z}
\! & =\! \tfrac{1}{z}
\left( \chi -\dot{y}^\alpha \chi _\alpha \right)
-\tfrac{1}{z~z_{\dot{y}^{\beta _0}}}
\left(
2z_x\bar{\chi }_{\beta _0}+z_{y^\alpha }\bar{\chi }_{\beta _0}^\alpha
\right) ,
\end{align*}
thus proving that the rank of $\mathcal{D}^1$ on $U(\alpha _0,\beta _0)$
is equal to the dimension of the tangent space to $J^1(q)$ at each point.
Hence the only first-order differential invariants are the constants.
\subsection{The Hessian metric\label{Hessian}}
Let $p^r\colon J^r(\mathbb{R},M)\to \mathbb{R}$,
$p^{r,r^\prime }\colon J^r(\mathbb{R},M)\to J^{r^\prime }(\mathbb{R},M)$,
$r>r^\prime $, be the canonical projections {of the jet bundles of
$p\colon \mathbb{R}\times M\to \mathbb{R}$.}

The map $p^{10}\colon J^1(\mathbb{R},M)\to J^0(\mathbb{R},M)
=\mathbb{R}\times M$ is an affine bundle modelled over the vector bundle
$W=p^\ast T^\ast \mathbb{R}\otimes V(p)$ $\cong \mathbb{R}\times TM$.
Hence each fibre $F_{x,y}=(p^{10})^{-1}(x,y)$, $(x,y)\in \mathbb{R}\times M$,
is an affine space modelled over $T_{y}M$. We set $\mathcal{L}^{x,y}=\left.
\mathcal{L}\right\vert _{F_{x,y}}$,
$\forall \mathcal{L}\in C^\infty (J^1(\mathbb{R},M))$.

Let $A$ be a real affine space of finite dimension modelled over a real vector
space $V$, endowed with its canonical $C^\infty $ structure. Every vector
$v\in V$ induces a vector field $\tilde{v}\in \mathfrak{X}(A)$ given by,
\[
\tilde{v}_x(f)=\left. \tfrac{d}{dt}\right\vert _{t=0}f(x+tv),
\quad
\forall x\in A,\;\forall  f\in C^\infty (A).
\]
If $v_1,\dotsc,v_n$ is a basis for $V$, then $\tilde{v}_1,\dotsc,\tilde{v}_n$
is a basis for the $C^\infty (A)$-module $\mathfrak{X}(A)$.

There exists a unique linear connection $D^A$ on $A$ such that,
$D^A\tilde{v}=0$, $\forall  v\in V$. This connection is symmetric and flat.

The image of a vector field $X\in \mathfrak{X}(M)$ by a diffeomorphism
$\varphi \colon M\to  M^\prime $ is the vector field
$\varphi \cdot X\in \mathfrak{X}(M^\prime )$ defined as follows:
$(\varphi \cdot X)_{x^\prime }=\varphi _{\ast}(X_{\varphi ^{-1}(x^\prime )})$,
$\forall  x^\prime \in M^\prime $.

If $D$ is a linear connection on $M$, then $\varphi \cdot D$
denotes the linear connection on $M^\prime $ defined
by the following formula:
\[
(\varphi \cdot D)_{X^\prime }Y^\prime
=\varphi \cdot
\left(
D_{\varphi ^{-1}\cdot X^\prime }
\left(
\varphi ^{-1}\cdot Y^\prime
\right)
\right) ,
\quad
\forall X^\prime ,Y^\prime \in \mathfrak{X}(M^\prime ).
\]
If $\omega $ is a $1$-form on an affine space $A$, then $D^A\omega $
is the covariant tensor of degree $2$ given by,
$\left( D^A\omega \right)
(X,Y)=\left( D_X^A\omega \right) (Y)=X\left( \omega (Y)\right)
-\omega \left( D_X^AY\right) $, $\forall X,Y\in \mathfrak{X}(A)$.
\begin{lemma}
\label{lemma_D}With the previous notations and definitions, for every
isomorphism of affine spaces $\alpha\colon A\to  A^\prime $ and every
$1$-form $\omega ^\prime $ on $A^\prime $ the following formulas hold:
\[
\alpha\cdot D^A=D^{A^\prime },\qquad\alpha^\ast (D^{A^\prime }
\omega ^\prime )=D^A\left( \alpha^\ast \omega ^\prime \right) .
\]
\end{lemma}
\begin{proof}
Actually, from the very definition of an affine morphism there exists
a linear isomorphism $\overrightarrow{\alpha }\colon V\in V^\prime $
such that,
$\alpha (v+a)=\overrightarrow{\alpha }(v)+\alpha (a)$,
$\forall a\in A$, $\forall v\in V$, and we have
$\widetilde{\overrightarrow{\alpha }(v)}=\alpha\cdot\tilde{v}$,
as follows from the next equalities:
\begin{align*}
\left( \alpha\cdot\tilde{v}\right) _{x^\prime }f^\prime
& =\left[
\alpha _\ast \left( \tilde{v}_{\alpha ^{-1}(x^\prime )}\right)
\right]
(f^\prime )\\
& =\tilde{v}_{\alpha^{-1}(x^\prime )}\left( f^\prime \circ\alpha\right) \\
& =\lim_{t\to 0}\frac{\left( f^\prime \circ \alpha \right)
\left( \alpha ^{-1}(x^\prime )+tv\right) -f^\prime (x^\prime )}{t}\\
& =\lim_{t\to 0}\frac{f^\prime (x^\prime +t\overrightarrow{\alpha }(v))
-f^\prime (x^\prime )}{t}\\
& =\widetilde{\overrightarrow{\alpha }(v)}_{x^\prime }(f^\prime ).
\end{align*}
By writing $u=\overrightarrow{\alpha }^{-1}(u^\prime )$,
$v=\overrightarrow{\alpha }^{-1}(v^\prime )$, for all
$u^\prime ,v^\prime \in V^\prime $, we obtain
\begin{align*}
\left( \alpha \cdot D^A\right) _{\widetilde{u^\prime }}
\left(
\widetilde{v^\prime }\right)
& =\left(
\alpha \cdot D^A
\right) _{\widetilde{\overrightarrow{\alpha}(u)}}
\left( \widetilde{\overrightarrow{\alpha }(v)}\right) \\
& =\left( \alpha \cdot D^A
\right) _{\alpha\cdot\tilde{u}}(\alpha \cdot \tilde{v})\\
& =0.
\end{align*}
If $X=\tilde{u}$, $Y=\tilde{v}$, then
$\left( D^A\omega\right) \left( \tilde{u},\tilde{v}\right)
=\tilde{u}\left( \omega\left( \tilde{v}\right) \right) $,
by virtue of the definition of the connection
$D^A$ and we deduce
\begin{align*}
\left( \alpha^\ast (D^{A^\prime }\omega ^\prime )\right) _a
\left( \tilde{u},\tilde{v}\right)
& =(D^{A^\prime }\omega ^\prime )_{\alpha (a)}
\left( \alpha_{\ast}\left( \tilde{u}_{a}\right) ,
\alpha _\ast \left( \tilde{v}_a\right) \right) \\
& =\left[ \left( \alpha \cdot D^A\right) (\omega ^\prime )
\right] _{\alpha (a)}\left(
\left( \alpha \cdot\tilde{u}
\right) _{\alpha(a)},\left( \alpha \cdot \tilde{v}
\right) _{\alpha(a)}\right) \\
& =\left[ \left( \alpha\cdot D^A\right)
(\omega ^\prime )_{\alpha \cdot \tilde{u}}
\left( \alpha \cdot \tilde{v}\right) \right]
(\alpha(a))\\
& =\left( \alpha \cdot \tilde{u}
\right) _{\alpha (a)}
\left( \omega ^\prime \left( \alpha \cdot \tilde{v}\right) \right)
-\omega ^\prime
\left( \alpha \cdot \left( D_{\tilde{u}}^A\tilde{v}\right) \right)
(\alpha (a))\\
& =\left( \alpha \cdot \tilde{u}\right) _{\alpha (a)}
\left( \omega ^\prime \left( \alpha\cdot\tilde{v}\right) \right) \\
& =\tilde{u}_a\left[ \omega ^\prime
\left( \alpha \cdot \tilde{v}\right)
\circ \alpha \right] \\
& =\tilde{u}\left[ \omega ^\prime \left( \alpha \cdot \tilde{v}\right)
\circ \alpha \right] (a)\\
& =\tilde{u}\left[ \left( \alpha ^\ast \omega ^\prime \right)
\left( \tilde{v}\right) \right] (a)\\
& =\left[ D^A\left( \alpha ^\ast \omega ^\prime \right)
\right] _a\left( \tilde{u},\tilde{v}\right) ,
\end{align*}
thus allowing one to conclude the proof.
\end{proof}

The Hessian metric of a function
$\mathcal{L\in}C^\infty (J^1(\mathbb{R},M))$
is the section of the vector bundle
$\varkappa \colon S^2\left[ V^\ast (p^{10})\right] \to J^1(\mathbb{R},M)$
defined as follows (cf.\ \cite[Definition 2.1]{Shima}):
\begin{equation}
\operatorname{Hess}\nolimits_{j_x^1\sigma }(\mathcal{L})
=D^{F_{x,\sigma (x)}}\left( d\mathcal{L}^{x,\sigma (x)}\right) .
\label{Hessian_metric}
\end{equation}
In local coordinates, $\operatorname{Hess}(\mathcal{L})
=\tfrac{\partial ^2\mathcal{L}}{\partial\dot{y}^\alpha \partial\dot{y}^\beta }
d_{10}\dot {y}^\alpha \otimes d_{10}\dot{y}^\beta $, where
$d_{10}{f}=\left. d{f}\right\vert _{V(p^{10})}$.

For every $\Phi \in \operatorname{Diff}(\mathbb{R}\times M)$, the diffeomorphism
$\Phi ^{(1)}\colon J^1(\mathbb{R},M)\to  J^1(\mathbb{R},M)$
transforms $V(p^{10})$ into itself because of the commutativity of the
following diagram:
\[
\begin{array}
[c]{ccc}
J^1(\mathbb{R},M) & \overset{\Phi ^{(1)}}{\longrightarrow} & J^1
(\mathbb{R},M)\\
\downarrow\scriptstyle{p}^{10} &  & \downarrow\scriptstyle{p}^{10}\\
\mathbb{R}\times M & \overset{\Phi }{\longrightarrow} & \mathbb{R}\times M
\end{array}
\]
Hence for every $\mathcal{L}\in C^\infty (J^1(\mathbb{R},M))$ the inverse
image of the Hessian metric $(\Phi ^{(1)})^\ast \operatorname{Hess}
(\mathcal{L})$ is another section of $S^2\left[ V^\ast (p^{10})\right]
\to  J^1(\mathbb{R},M)$.
 \begin{proposition}
\label{covv}For every $\Phi \in \operatorname{Aut}(p)$ and every
$\mathcal{L}\in C^\infty (J^1(\mathbb{R},M))$, we have
\[
(\Phi ^{(1)})^\ast \operatorname{Hess}(\mathcal{L})=\operatorname{Hess}
(\mathcal{L\circ}\Phi ^{(1)}).
\]
\end{proposition}
\begin{proof}
If $\Phi (x,y)=(\varphi (x),\Psi(x,y))$,
$\forall (x,y)\in \mathbb{R}\times M$,
then, taking account of the fact that the affine bundle
$J^1(\mathbb{R},M)\to \mathbb{R}\times M$ is modelled over the vector bundle
$\mathbb{R}\times TM\to \mathbb{R}\times M$ it follows that
$\Phi ^{(1)}\colon J^1(\mathbb{R},M)\to  J^1(\mathbb{R},M)$ is an affine
morphism whose associated linear morphism
$\vec{\Phi }^{(1)}\colon \mathbb{R}\times TM\to \mathbb{R}\times TM$
is given by,
$\vec{\Phi }^{(1)}(x,v)=\left( x,\left( \Psi_x\right) _\ast v\right) $,
$\forall v\in T_yM$, and the statement is a consequence of Lemma \ref{lemma_D},
as
\begin{align*}
(\Phi ^{(1)})^\ast \operatorname{Hess}(\mathcal{L})
& =(\Phi ^{(1)})^\ast D^{F_{x,\sigma (x)}}
\left( d\mathcal{L}^{x,\sigma (x)}\right) \\
& =D^{F_{x,\sigma (x)}}
\left( (\Phi ^{(1)})^\ast d\mathcal{L}^{x,\sigma (x)}\right) \\
& =D^{F_{x,\sigma (x)}}
d\left( (\Phi ^{(1)})^\ast \mathcal{L}^{x,\sigma (x)}\right) \\
& =D^{F_{x,\sigma (x)}}
d\left( \mathcal{L}^{x,\sigma (x)}\circ \Phi ^{(1)}\right) \\
& =\operatorname{Hess}(\mathcal{L\circ}\Phi ^{(1)}).
\end{align*}
\end{proof}
\subsection{Second order}
\subsubsection{The basic invariant\label{basic_invariant}}
Let $O^2\subset J^2(q)$ be the dense open subset of elements $j_{j_{x_0
}^1\!\sigma _0}^2(\mathcal{L})$ for which the Hessian metric
$\operatorname{Hess}_{j_{x_0}^1\!\sigma _0}(\mathcal{L})$ is
non-singular. In coordinates, $O^2$ is defined by the inequation
$\det\left( z_{\dot{y}^\alpha \dot{y}^\beta }\right) _{\alpha,\beta=1}
^{m}\neq 0$. Hence, for every $j_{j_{x_0}^1\sigma _0}^2(\mathcal{L})\in
O^2$ the linear mapping
\[
\begin{array}
[c]{l}
\operatorname{Hess}\nolimits_{j_{x_0}^1\! \sigma _0}
(\mathcal{L})^\flat \colon V_{j_{x_0}^1\! \sigma _0}(p^{10})
\to V_{j_{x_0}^1\!\sigma _0}^\ast (p^{10}),
\medskip \\
\operatorname{Hess}\nolimits_{j_{x_0}^1\! \sigma _0}
(\mathcal{L})^\flat (X)(Y)
=\operatorname{Hess}\nolimits_{j_{x_0}^1\! \sigma _0}
(\mathcal{L})(X,Y),
\end{array}
\]
is an isomorphism, the inverse of which is denoted by
\[
\operatorname{Hess}\nolimits_{j_{x_0}^1\! \sigma _0}
(\mathcal{L})^\sharp \colon V_{j_{x_0}^1\! \sigma _0}^\ast
(p^{10})\to  V_{j_{x_0}^1\! \sigma _0}(p^{10}).
\]
A contravariant metric in $S^2V_{j_{x_0}^1\! \sigma _0}(p^{10})$
is then defined as follows:
\[
^{\sharp}\! \operatorname{Hess}\nolimits_{j_{x_0}^1\! \sigma _0}
(\mathcal{L})(w_1,w_2)\! =\! \operatorname{Hess}\nolimits_{j_{x_0}^1
\! \sigma _0}(\mathcal{L})(\operatorname{Hess}\nolimits_{j_{x_0}^1
\! \sigma _0}(\mathcal{L})^\sharp (w_1),
\operatorname{Hess}\nolimits_{j_{x_0}^1\! \sigma _0}
(\mathcal{L})^\sharp (w_2)),
\]
for all $w_1,w_{2}\in V_{j_{x_0}^1\!\sigma _0}^\ast (p^{10})$.
\begin{proposition}
\label{proposition_V_I}
With the same notations as above, let $V\colon O^2\to \mathbb{R}$
be the function defined by, $V(j_{j_x^1\sigma }^2\mathcal{L)}
=~^{\sharp}\operatorname{Hess}\nolimits_{j_x^1\sigma }
(\mathcal{L})\left( d_{10}\mathcal{L},d_{10}\mathcal{L}\right) $.

For every
$\Phi \in \operatorname{Aut}(p)$ and all
$\mathcal{L}\in C^\infty (J^1(\mathbb{R},M))$, the following formula holds:
\[
V\left\{  ((\tilde{\Phi}^{-1})^{(1)})^{(2)}
(j_{\Phi^{(1)}(j_{x}^{1}\sigma )}^{2}\mathcal{L})\right\}
=(\phi^{\prime})^{-1}V(j_{\Phi^{(1)}
(j_{x}^{1}\sigma)}^{2}\mathcal{L}).
\]
Therefore, the function $V$ is invariant under the natural action on
$O^2\subset J^2(q)$ of the normal subgroup $\operatorname{Aut}^{v}(p)
\subset\operatorname{Aut}(p)$ of automorphisms of $p$ inducing the
identity on the real line (the so-called vertical group). Moreover,
if $O^{\prime2}$ is the dense open subset of $2$-jets $j_{j_x^1\sigma }
^2(\mathcal{L)\in}O^2 $ such that $\mathcal{L}(j_x^1\sigma )\neq 0$,
then the function $I\colon O^{\prime 2}\to \mathbb{R}$ defined by,
\[
I(j_{j_x^1\sigma }^2\mathcal{L)}=\frac{V(j_{j_x^1\sigma }^2
\mathcal{L)}}{\mathcal{L}\left( j_x^1\sigma \right) }
\]
is invariant under the full group $\operatorname{Aut}(p)$
of automorphisms of $p$.
\end{proposition}
\begin{proof}
If $a=j_x^1\sigma $, $a^\prime =\Phi ^{(1)}(j_x^1\sigma )$, then
\begin{align*}
V\left\{
((\tilde{\Phi }^{-1})^{(1)})^{(2)}(j_{a^\prime }^2s_{\mathcal{L}})
\right\}
& =V\left\{
j_a^2\left( (\tilde{\Phi }^{-1})
\circ s_{\mathcal{L}}\circ \Phi ^{(1)})\right)
\right\} \\
& =V\left\{ j_a^2s_{\mathcal{\bar{L}}}\right\}
\end{align*}
with
$\mathcal{\bar{L}}=(\phi^{\prime})^{-1}(\mathcal{L}\circ \Phi ^{(1)})$,
where $\phi ^{\prime }=d\phi / dx$, and hence
\[
V\left\{
((\tilde{\Phi }^{-1})^{(1)})^{(2)}(j_{a^\prime }^2s_{\mathcal{L}})
\right\}
=\mathrm{\operatorname{Hess}}_a(\mathcal{\bar{L})}
(\mathrm{\operatorname{Hess}}_a(\mathcal{\bar{L})}^\sharp (d_{10}
\mathcal{\bar{L}}),\mathrm{\operatorname{Hess}}_a
(\mathcal{\bar{L})}^\sharp (d_{10}\mathcal{\bar{L}})).
\]
We first note that the following formulas hold:
\begin{align*}
\operatorname{Hess}_a(\mathcal{\bar{L})}
& =(\phi ^\prime )^{-1}\operatorname{Hess}_a\mathcal{(L}\circ \Phi ^{(1)}),\\
\operatorname{Hess}_a(\mathcal{\bar{L})}^\sharp
& =\phi ^\prime \operatorname{Hess}_a(\mathcal{L}\circ \Phi ^{(1)})^\sharp ,\\
d_{10}\mathcal{\bar{L}}
& =(\phi ^\prime )^{-1}d_{10}\mathcal{(L}\circ \Phi ^{(1)}),
\end{align*}
as $(\phi ^\prime )^{-1}$ does not depend on vertical variables.
Then, from the bilinearity of the Hessian we have
\begin{equation}
\begin{array}{l}
V\left\{
((\tilde{\Phi}^{-1})^{(1)})^{(2)}(j_{a^\prime }^2s_{\mathcal{L}})
\right\}
=(\phi ^{\prime })^{-1}\operatorname{Hess}_a
\mathcal{(L}\circ \Phi ^{(1)})\left( U,U\right) ,
\medskip  \\
U=\operatorname{Hess}_a\mathcal{(L}\circ \Phi ^{(1)})^\sharp
(d_{10}(\mathcal{L}\circ \Phi ^{(1)}).
\end{array}
\label{ccont}
\end{equation}

Moreover, from the covariance of the Hessian,
$\operatorname{Hess}(\mathcal{L}\circ \Phi ^{(1)})
=(\Phi ^{(1)})^\ast \operatorname{Hess}\mathcal{L}$
(see Proposition \ref{covv}), we conclude that its sharp operator
is also covariant, namely,
\[
\operatorname{Hess}_a(\mathcal{L}\circ \Phi ^{(1)})^\sharp
=((\Phi ^{-1})^{(1)})_\ast \circ \operatorname{Hess}_{a^\prime }
(\mathcal{L)}^\sharp \circ ((\Phi ^{-1}\mathcal{)}^{(1)})^\ast .
\]
In addition, $d_{10}(\mathcal{L}\circ \Phi ^{(1)})
=(\Phi ^{(1)})^\ast d_{10}\mathcal{L}$ as $(\Phi ^{(1)})_\ast $
transforms vertical vectors into vertical vectors.
From these two facts and \eqref{ccont} we have
\begin{multline*}
V\left\{
((\tilde{\Phi }^{-1})^{(1)})^{(2)}
(j_{a^\prime }^2s_{\mathcal{L}})
\right\} = \\
(\phi ^\prime )^{-1}\operatorname{Hess}_{a^\prime }
\mathcal{(L})(\operatorname{Hess}_{a^\prime }
\mathcal{(L})^\sharp (d_{10}\mathcal{L}),
\operatorname{Hess}_{a^\prime }\mathcal{(L})^\sharp (d_{10}\mathcal{L})),
\end{multline*}
and we obtain the first formula in the statement.
\end{proof}
\subsubsection{The generic rank of $\mathcal{D}^2$ computed}
The coordinate system induced by $(x,y^\alpha )$\ on $J^2(q)$ is
\[
\begin{array}
[c]{ll}
x,z,z_x,z_{xx}, & \\
y^\alpha ,\dot{y}^\alpha ,z_{y^\alpha },z_{\dot{y}^\alpha },
z_{xy^\alpha },z_{x\dot{y}^\alpha }, & 1\leq \alpha \leq m,\\
z_{y^\alpha y^\beta },z_{\dot{y}^\alpha \dot{y}^\beta },
& 1\leq \alpha \leq \beta \leq m,\\
z_{y^\alpha \dot{y}^\beta }, & \alpha ,\beta =1,\dotsc,m.
\end{array}
\]
Hence, $\dim J^2(q)=2m^2+7m+4$.
\begin{theorem}
\label{theorem_rank_D2}
On a dense open subset $O^{\prime 2}\subset O^2\subset J^2(q)$,
where $O^2$ is the set of $2$-jets whose Hessian metric is
non-singular, the rank of the distribution $\mathcal{D}^2$
generated by all the vector fields of the form
$(\tilde{X}^{(1)})^{(2)}$, $X$ being an arbitrary $p$-projectable
vector field on $\mathbb{R}\times M$, is $2m^2+7m+3=(m+3)(2m+1)$.
Consequently, the invariant $I$ defined in
\emph{Proposition \ref{proposition_V_I}} is a basis
for the invariants of second order.
\end{theorem}
\begin{proof}
We first compute
\begin{align*}
(\tilde{X}^{(1)})^{(2)} & =u\tfrac{\partial }{\partial x}
+v^\alpha \tfrac{\partial }{\partial y^\alpha }
+v_1^\alpha \tfrac{\partial }{\partial\dot{y}^\alpha }
-\tfrac{du}{dx}z\tfrac{\partial }{\partial z}
+A\tfrac{\partial }{\partial z_x}
+B_{\alpha}\tfrac{\partial }{\partial z_{y^\alpha }}\\
& +C_{\alpha}\tfrac{\partial }{\partial z_{\dot{y}^\alpha }}
+D_{xx}\tfrac{\partial }{\partial z_{xx}}
+E_{xy^\alpha }\tfrac{\partial }{\partial z_{xy^\alpha }}
+F_{x\dot{y}^\alpha }
\tfrac{\partial }{\partial z_{x\dot{y}^\alpha }}\\
& +\sum\nolimits_{\alpha \leq \beta}G_{y^\alpha y^\beta }
\tfrac{\partial }{\partial z_{y^\alpha y^\beta }}
+H_{y^\alpha \dot{y}^\beta }
\tfrac{\partial }{\partial z_{y^\alpha \dot{y}^\beta }}
+\sum\nolimits_{\alpha \leq \beta }
K_{\dot{y}^\alpha \dot{y}^\beta }
\tfrac{\partial }{\partial z_{\dot{y}^\alpha \dot{y}^\beta }},
\end{align*}
where $A$, $B_\alpha $, and $C_\alpha $ are given by the formulas
\eqref{A}, \eqref{B}, and \eqref{C}, respectively,
and the coefficients $D_{xx}$, $E_{xy^\alpha }$,
$F_{x\dot{y}^\alpha }$, $G_{y^\alpha y^\beta }$,
$H_{y^\alpha \dot{y}^\beta }$, $K_{\dot{y}^\alpha \dot{y}^\beta }$
are to be determined by using the formulas \eqref{Yr},
and \eqref{eta_I_bis}. We obtain
\begin{align}
D_{xx} & =-3\tfrac{du}{dx}z_{xx}
-2\tfrac{\partial v^\alpha }{\partial x}z_{xy^\alpha }
-3\tfrac{d^2u}{dx^2}z_x
+2\tfrac{d^2u}{dx^2}\dot{y}^\alpha z_{x\dot{y}^\alpha }
\label{D}\\
& -\tfrac{\partial^2v^\alpha }{\partial x^2}z_{y^\alpha }
-2\tfrac{\partial^2v^\alpha }{\partial x^2}z_{x\dot{y}^\alpha }
-2\tfrac{\partial^2v^\alpha }{\partial x\partial y^\beta }
\dot{y}^\beta z_{x\dot{y}^\alpha }
\nonumber\\
& -\tfrac{d^3u}{dx^3}z
+\tfrac{d^3u}{dx^3}\dot{y}^\alpha z_{\dot{y}^\alpha }
-\tfrac{\partial^3v^\alpha }{\partial x^3}z_{\dot{y}^\alpha }
-\tfrac{\partial^3v^\alpha }{\partial x^2\partial y^\beta}
\dot{y}^\beta z_{\dot{y}^\alpha },
\nonumber
\end{align}
\begin{align}
E_{xy^\alpha } & =-2\tfrac{du}{\partial x}z_{xy^\alpha }
-\tfrac{\partial v^\beta }{\partial x}
z_{y^\alpha y^\beta }
-\tfrac{\partial v^\beta }{\partial y^\alpha }
z_{xy^\beta } \label{E}\\
& -\tfrac{d^2u}{dx^2}z_{y^\alpha }
+\tfrac{d^2u}{dx^2}\dot{y}^\beta z_{y^\alpha \dot{y}^\beta }
-\tfrac{\partial^2v^\beta }{\partial x^2}
z_{y^\alpha \dot{y}^\beta }\nonumber\\
& -\tfrac{\partial^2v^\beta }{\partial x\partial y^\alpha }
z_{y^\beta }
-\tfrac{\partial^2v^\beta }{\partial x\partial y^\alpha }
z_{x\dot{y}^\beta }
-\tfrac{\partial^2v^\beta }{\partial x\partial y^\gamma }
\dot{y}^\gamma z_{y^\alpha \dot{y}^\beta }\nonumber\\
& -\tfrac{\partial^2v^\beta }{\partial y^\alpha \partial y^\gamma }
\dot{y}^\gamma z_{x\dot{y}^\beta }
-\tfrac{\partial^3v^\beta }{\partial x^2\partial y^\alpha }
z_{\dot{y}^\beta }
-\tfrac{\partial^3v^\beta }
{\partial x\partial y^\alpha \partial y^\gamma }
\dot{y}^\gamma z_{\dot{y}^\beta },\nonumber
\end{align}
\begin{align}
F_{x\dot{y}^\alpha }
& =-\tfrac{du}{dx}z_{\dot{y}^\alpha x}
-\tfrac{\partial v^\beta }{\partial x}
z_{y^\beta \dot{y}^\alpha }
-\tfrac{\partial v^\beta }{\partial y^\alpha }
z_{x\dot{y}^\beta }
\label{F}\\
& +\dot{y}^\beta
\tfrac{d^2u}{dx^2}z_{\dot{y}^\alpha \dot{y}^\beta }
-\tfrac{\partial^2v^\beta }{\partial x^2}
z_{\dot{y}^\alpha \dot{y}^\beta }-\dot{y}^\gamma
\tfrac{\partial^2v^\beta }{\partial x\partial y^\gamma }
z_{\dot{y}^\alpha \dot{y}^\beta }
-\tfrac{\partial^2v^\beta }{\partial x\partial y^\alpha }
z_{\dot{y}^\beta },
\nonumber
\end{align}
\begin{align}
G_{y^\alpha y^\beta }
& =-\tfrac{\partial^2v^\sigma }
{\partial y^\alpha \partial y^\beta }
z_{y^\sigma }
-\left(
\tfrac{\partial ^3v^\sigma }
{\partial x\partial y^\alpha \partial y^\beta }
+\tfrac{\partial^3v^\sigma }
{\partial y^\alpha \partial y^\beta \partial y^\gamma }
\dot{y}^\gamma
\right)
z_{\dot{y}^\sigma }
\label{G}\\
& -\tfrac{du}{dx}z_{y^\alpha y^\beta }
-\tfrac{\partial v^\sigma }{\partial y^\beta }
z_{y^\alpha y^\sigma }
-\tfrac{\partial v^\sigma }{\partial y^\alpha }
z_{y^\beta y^\sigma }\nonumber\\
& -\left(
\tfrac{\partial ^2v^\sigma }{\partial x\partial y^\beta }
+\tfrac{\partial ^2v^\sigma }{\partial y^\beta \partial y^\gamma }
\dot{y}^\gamma
\right)
z_{y^\alpha \dot{y}^\sigma }
-\left(
\tfrac{\partial^2v^\sigma }{\partial x\partial y^\alpha }
+\tfrac{\partial ^2v^\sigma }{\partial y^\alpha \partial y^\gamma }
\dot{y}^\gamma \right) z_{y^\beta \dot{y}^\sigma },\nonumber
\end{align}
\begin{equation}
H_{y^\alpha \dot{y}^\beta }
=-\tfrac{\partial^2v^\sigma }{\partial y^\alpha \partial y^\beta }
z_{\dot{y}^\sigma }-\tfrac{\partial v^\sigma }{\partial y^\beta }
z_{y^\alpha \dot{y}^\sigma }
-\tfrac{\partial v^\sigma }{\partial y^\alpha }
z_{y^\sigma \dot{y}^\beta }
-\left(
\tfrac{\partial^2v^\sigma }{\partial x\partial y^\alpha }
+\tfrac{\partial ^2v^\sigma }{\partial y^\alpha \partial y^\gamma }
\dot{y}^\gamma
\right)
z_{\dot{y}^\beta \dot{y}^\sigma }, \label{H}
\end{equation}
\begin{equation}
K_{\dot{y}^\alpha \dot{y}^\beta }
=\tfrac{du}{dx}z_{\dot{y}^\alpha \dot{y}^\beta }
-\tfrac{\partial v^\sigma }{\partial y^\beta }
z_{\dot{y}^\alpha \dot{y}^\sigma }
-\tfrac{\partial v^\sigma }{\partial y^\alpha }
z_{\dot{y}^\beta \dot{y}^\sigma }. \label{K}
\end{equation}
Hence, taking the formula for $(\tilde{X}^{(1)})^{(1)}$,
and the expressions \eqref{D}, \eqref{E}, \eqref{F},
\eqref{G}, \eqref{H}, and \eqref{K} for $D_{xx}$,
$E_{xy^\alpha }$, $F_{x\dot{y}^\alpha }$,
$G_{y^\alpha y^\beta }$, $H_{y^\alpha \dot{y}^\beta }$,
$K_{\dot{y}^\alpha \dot{y}^\beta }$ into account, we have
\begin{align}
\quad(\tilde{X}^{(1)})^{(2)}
& =u\tfrac{\partial }{\partial x}
+v^\alpha \tfrac{\partial }{\partial y^\alpha }
+\tfrac{du}{dx}\chi_1^1
+\tfrac{\partial v^\alpha }{\partial x}
\bar{\chi }_\alpha ^1
+\tfrac{\partial v^\alpha }{\partial y^\beta }
\tilde{\chi }_\alpha ^\beta
+\tfrac{d^2u}{dx^2}\chi _1^{11} \label{X1^2}\\
& +\tfrac{\partial^2v^\alpha }{\partial x^2}
\bar{\chi }_\alpha ^{11}
+\tfrac{\partial^2v^\alpha }{\partial x\partial y^\beta }
\tilde{\chi }_\alpha^{\beta 1}
+\sum\nolimits_{\alpha \leq \beta}
\tfrac{\partial ^2v^\sigma }{\partial y^\alpha \partial y^\beta }
\hat{\chi }_\sigma ^{\alpha \leq \beta}
+\tfrac{d^3u}{dx^3}\chi_1^{111}\nonumber\\
& +\tfrac{\partial^3v^\alpha }{\partial x^3}
\bar{\chi }_\alpha ^{111}
+\tfrac{\partial^3v^\beta }{\partial x^2\partial y^\alpha }
\tilde{\chi}_{\beta}^{11\alpha}
+\sum\nolimits_{\alpha\leq\beta}
\tfrac{\partial ^3v^\sigma }
{\partial x\partial y^\alpha \partial y^\beta }
\chi _\sigma ^{1,\alpha \leq \beta } \nonumber\\
& +\sum\nolimits_{\alpha \leq \beta \leq \gamma }
\tfrac{\partial ^3v^\sigma }
{\partial y^\alpha \partial y^\beta \partial y^\gamma }
\chi _\sigma ^{\alpha \leq \beta \leq \gamma},\nonumber
\end{align}
where
\begin{align}
\chi _1^1 & =-z\tfrac{\partial }{\partial z}
-\dot{y}^\alpha
\tfrac{\partial }{\partial\dot{y}^\alpha }
-2z_x\tfrac{\partial }{\partial z_x}-z_{y^\alpha }
\tfrac{\partial }{\partial z_{y^\alpha }}-3z_{xx}
\tfrac{\partial }{\partial z_{xx}}-2z_{xy^\alpha }
\tfrac{\partial }{\partial z_{xy^\alpha }}
\label{chi11}\\
& -z_{x\dot{y}^\alpha }
\tfrac{\partial }{\partial z_{x\dot{y}^\alpha }}
-\sum\nolimits_{\alpha\leq\beta}z_{y^\alpha y^\beta }
\tfrac{\partial }{\partial z_{y^\alpha y^\beta }}
+\sum\nolimits_{\alpha\leq\beta}z_{\dot{y}^\alpha \dot{y}^\beta }
\tfrac{\partial }{\partial z_{\dot{y}^\alpha \dot{y}^\beta }},
\nonumber
\end{align}
\begin{equation}
\bar{\chi }_\alpha ^1
=\tfrac{\partial }{\partial\dot{y}^\alpha }
-z_{y^\alpha }\tfrac{\partial }{\partial z_x}-2z_{xy^\alpha }
\tfrac{\partial }{\partial z_{xx}}
-z_{y^\alpha y^\beta }
\tfrac{\partial }{\partial z_{xy^\beta }}
-z_{y^\alpha \dot{y}^\beta }
\tfrac{\partial}{\partial z_{x\dot{y}^\beta }},
\label{chibarra1alfa}
\end{equation}
\begin{align}
\tilde{\chi }_{\alpha}^\beta
& =\dot{y}^\beta \tfrac{\partial }{\partial\dot{y}^\alpha }
-z_{y^\alpha }\tfrac{\partial }{\partial z_{y^\beta }}
-z_{\dot{y}^\alpha }
\tfrac{\partial }{\partial z_{\dot{y}^\beta }}
-z_{xy^\alpha }\tfrac{\partial }{\partial z_{xy^\beta }}
-z_{x\dot{y}^\alpha }
\tfrac{\partial }{\partial z_{x\dot{y}^\beta }
}\label{chitildebetaalfa}\\
& -(1+\delta _\sigma ^\beta )z_{y^\sigma y^\alpha }
\tfrac{\partial }{\partial z_{y^\sigma y^\beta }}
-z_{y^\sigma \dot{y}^\alpha }
\tfrac{\partial }{\partial z_{y^\sigma \dot{y}^\beta }}
-z_{y^\alpha \dot{y}^\sigma }
\tfrac{\partial }{\partial z_{y^\beta \dot{y}^\sigma }
}\nonumber\\
& -(1+\delta_{\sigma }^\beta )
z_{\dot{y}^\alpha \dot{y}^\sigma }
\tfrac{\partial }{\partial z_{\dot{y}^\sigma \dot{y}^\beta }},
\nonumber
\end{align}
\begin{align}
\chi _1^{11}
& =\left( -z+\dot{y}^\gamma z_{\dot{y}^\gamma }\right)
\tfrac{\partial }{\partial z_x}
-3z_x\tfrac{\partial }{\partial z_{xx}}
+2\dot{y}^\alpha z_{x\dot{y}^\alpha }
\tfrac{\partial }{\partial z_{xx}}
-z_{y^\alpha }\tfrac{\partial }{\partial z_{xy^\alpha }}
\label{chi111}\\
& +\dot{y}^\beta z_{y^\alpha \dot{y}^\beta }
\tfrac{\partial }{\partial z_{xy^\alpha }}
+\dot{y}^\beta z_{\dot{y}^\alpha \dot{y}^\beta }
\tfrac{\partial }{\partial z_{x\dot{y}^\alpha }},
\nonumber
\end{align}
\begin{equation}
\bar{\chi }_\alpha ^{11}
=-z_{\dot{y}^\alpha }
\tfrac{\partial }{\partial z_x}
-z_{y^\alpha }
\tfrac{\partial }{\partial z_{xx}}-2z_{x\dot{y}^\alpha }
\tfrac{\partial }{\partial z_{xx}}
-z_{y^\beta \dot{y}^\alpha }
\tfrac{\partial }{\partial z_{xy^\beta }}
-z_{\dot{y}^\beta \dot{y}^\alpha }
\tfrac{\partial }{\partial z_{x\dot{y}^\beta }},
\label{chibarra11alfa}
\end{equation}
\begin{align}
\tilde{\chi }_\alpha ^{\beta 1}
& =-\dot{y}^\beta z_{\dot{y}^\alpha }
\tfrac{\partial }{\partial z_x}
-z_{\dot{y}^\alpha }\tfrac{\partial }{\partial z_{y^\beta }}
-2\dot{y}^\beta z_{x\dot{y}^\alpha }
\tfrac{\partial }{\partial z_{xx}}
-z_{y^\alpha }\tfrac{\partial }{\partial z_{xy^\beta }}
\label{chitildebeta1alfa}\\
& -z_{x\dot{y}^\alpha }\tfrac{\partial }{\partial z_{xy^\beta }}
-\dot{y}^\beta z_{y^\sigma \dot{y}^\alpha }
\tfrac{\partial }{\partial z_{xy^\sigma }}
-\dot{y}^\beta z_{\dot{y}^\sigma \dot{y}^\alpha }
\tfrac{\partial }{\partial z_{x\dot{y}^\sigma }}
-z_{\dot{y}^\alpha }
\tfrac{\partial }{\partial z_{x\dot{y}^\beta }}
\nonumber\\
& -(1+\delta _{\beta \gamma })z_{y^\gamma \dot{y}^\alpha }
\tfrac{\partial }{\partial z_{y^\gamma y^\beta }}
-z_{\dot{y}^\sigma \dot{y}^\alpha }
\tfrac{\partial }{\partial z_{y^\beta \dot{y}^\sigma }},
\nonumber
\end{align}
\begin{align}
\quad
\hat{\chi }_\sigma ^{\alpha \leq \beta}
\!\! & =\!\!-\tfrac{1}{1+\delta _{\alpha \beta}}
z_{\dot{y}^\sigma }
\left(
\dot{y}^\beta
\tfrac{\partial }{\partial z_{y^\alpha }}
+\dot{y}^\alpha \tfrac{\partial }{\partial z_{y^\beta }}
\right)
\label{chigorroalgagamabeta}\\
& \!\! -\tfrac{1}{1+\delta _{\alpha \beta}}
z_{x\dot{y}^\sigma }
\left(
\dot{y}^\beta \tfrac{\partial }{\partial z_{xy^\alpha }}
+\dot{y}^\alpha \tfrac{\partial }{\partial z_{xy^\beta }}
\right)
-z_{y^\sigma }
\tfrac{\partial }{\partial z_{y^\alpha y^\beta }}
\nonumber\\
& \!\!-\tfrac{1}{1+\delta _{\alpha \beta }}
\left( (1+\delta _{\beta \gamma })
\dot{y}^\alpha z_{y^\gamma \dot{y}^\sigma }
\tfrac{\partial }{\partial z_{y^\gamma y^\beta }}
+(1+\delta _{\alpha \gamma })
\dot{y}^\beta z_{y^\gamma \dot{y}^\sigma }
\tfrac{\partial }{\partial z_{y^\gamma y^\alpha }}
\right)
\nonumber\\
& \!\!-\tfrac{1}{1+\delta _{\alpha \beta }}z_{\dot{y}^\sigma }
\left(
\tfrac{\partial }{\partial z_{y^\alpha \dot{y}^\beta }}
+\tfrac{\partial }{\partial z_{y^\beta \dot{y}^\alpha }}
\right)
\nonumber\\
& \!\!-\tfrac{1}{1+\delta _{\alpha \beta }}
z_{\dot{y}^\gamma \dot{y}^\sigma }
\left(
\dot{y}^\beta
\tfrac{\partial }{\partial z_{y^\alpha \dot{y}^\gamma }}
+\dot{y}^\alpha
\tfrac{\partial }{\partial z_{y^\beta \dot{y}^\gamma }}
\right) ,
\nonumber
\end{align}
\begin{align}
\chi _1^{111}
& =\left( \dot{y}^\alpha z_{\dot{y}^\alpha }-z\right)
\tfrac{\partial }{\partial z_{xx}},
\label{chi1111}\\
\bar{\chi }_\alpha ^{111}
& =-z_{\dot{y}^\alpha }\tfrac{\partial }{\partial z_{xx}},
\label{chibarra111alfa}\\
\tilde{\chi }_\beta ^{11\alpha}
& =-\dot{y}^\alpha z_{\dot{y}^\beta }
\tfrac{\partial }{\partial z_{xx}}
-z_{\dot{y}^\beta }\tfrac{\partial }{\partial z_{xy^\alpha }},
\label{chitilde11alfabeta}\\
\chi _\sigma ^{1,\alpha\leq\beta}
& =-\tfrac{z_{\dot{y}^\sigma }}{1+\delta_{\alpha\beta}}
\left(
\dot{y}^\beta \tfrac{\partial }{\partial z_{xy^\alpha }}
+\dot{y}^\alpha \tfrac{\partial }{\partial z_{xy^\beta }}
\right)
-z_{\dot{y}^\sigma }
\tfrac{\partial }{\partial z_{y^\alpha y^\beta }},
\label{chi1alfabetasigma}
\end{align}
\begin{equation}
\chi _\sigma ^{\alpha \leq \beta \leq \gamma }
=-z_{\dot{y}^\sigma }
\left(
\dot{y}^\gamma
\tfrac{\partial }{\partial z_{y^\alpha y^\beta }}
+(1-\delta _{\beta \gamma })\dot{y}^\beta
\tfrac{\partial }{\partial z_{y^\alpha y^\gamma }}
+(1-\delta_{\alpha\gamma})\dot{y}^\alpha
\tfrac{\partial }{\partial z_{y^\gamma y^\beta }}
\right) .
\label{chialfabetagamasigma}
\end{equation}

Therefore the vector fields $\tfrac{\partial }{\partial x}$,
$\tfrac{\partial }{\partial y^\alpha }$, and
\eqref{chi11}--\eqref{chialfabetagamasigma} span
the distribution $\mathcal{D}^2$. Let us fix four indices
$\alpha _0$, $\beta _0$, $\gamma _0$, $\sigma _0$.
From \eqref{chibarra111alfa} on the dense open subset
$z_{\dot{y}^{\alpha _0}}\neq 0$, we have
$\tfrac{\partial }{\partial z_{xx}}
=-\frac{1}{z_{\dot{y}^{\alpha _0}}}\bar{\chi }_{\alpha _0 }^{111}$.
Replacing this expression into \eqref{chitilde11alfabeta},
on the dense open subset $z_{\dot{y}^{\alpha _0}}\neq 0$,
$z_{\dot{y}^{\beta _0}}\neq 0$, we have
$\tfrac{\partial }{\partial z_{xy^\alpha }}
=\tfrac{\dot{y}^\alpha }{z_{\dot{y}^{\alpha_0}}}
\bar{\chi }_{\alpha_0}^{111}-\tfrac{1}{z_{\dot{y}^{\beta_0}}}
\tilde{\chi }_{\beta_0}^{11\alpha}$. From
\eqref{chialfabetagamasigma} we have
$\tfrac{\partial }{\partial z_{y^\alpha y^\beta }}
=-\tfrac{1}{\dot{y}^{\gamma _0}z_{\dot{y}^{\sigma _0}}}
\chi _{\sigma _0}^{\alpha \leq \beta ,\gamma _0}$
on the dense open subset $\dot{y}^{\gamma _0}\neq 0$,
$z_{\dot{y}^{\sigma _0}}\neq 0$. Hence the distribution
$\mathcal{D}^2$ is spanned by $\tfrac{\partial }{\partial x}$,
$\tfrac{\partial }{\partial y^\alpha }$,
$\tfrac{\partial }{\partial z_{xx}}$,
$\tfrac{\partial }{\partial z_{xy^\alpha }}$,
$\tfrac{\partial }{\partial z_{y^\alpha y^\beta }}$,
$\alpha\leq\beta$, and the following vector fields:
\begin{align*}
\chi _1^{\prime1}  & =-z\tfrac{\partial }{\partial z}
-\dot{y}^\alpha
\tfrac{\partial }{\partial\dot{y}^\alpha }
-2z_x\tfrac{\partial }{\partial z_x}
-z_{y^\alpha }
\tfrac{\partial }{\partial z_{y^\alpha }}
-z_{x\dot{y}^\alpha }
\tfrac{\partial }{\partial z_{x\dot{y}^\alpha }}\\
& +\sum\nolimits_{\alpha \leq \beta }
z_{\dot{y}^\alpha \dot{y}^\beta }
\tfrac{\partial }{\partial z_{\dot{y}^\alpha \dot{y}^\beta }},
\end{align*}
\begin{align}
\bar{\chi }_\alpha^{\prime 1}
& =\tfrac{\partial }{\partial\dot{y}^\alpha }
-z_{y^\alpha }\tfrac{\partial }{\partial z_x}
-z_{y^\alpha \dot{y}^\beta }
\tfrac{\partial }{\partial z_{x\dot{y}^\beta }},
\nonumber\\
\tilde{\chi }_\alpha ^{\prime \beta}
& =\dot{y}^\beta \tfrac{\partial }{\partial\dot{y}^\alpha }
-z_{y^\alpha }
\tfrac{\partial }{\partial z_{y^\beta }}
-z_{\dot{y}^\alpha }
\tfrac{\partial }{\partial z_{\dot{y}^\beta }}
-z_{x\dot{y}^\alpha }
\tfrac{\partial }{\partial z_{x\dot{y}^\beta }}
\nonumber\\
& -z_{y^\sigma \dot{y}^\alpha }
\tfrac{\partial }{\partial z_{y^\sigma \dot{y}^\beta }}
-z_{y^\alpha \dot{y}^\sigma }
\tfrac{\partial }{\partial z_{y^\beta \dot{y}^\sigma }}
-(1+\delta _\sigma ^\beta )z_{\dot{y}^{\alpha }\dot{y}^\sigma }
\tfrac{\partial }{\partial z_{\dot{y}^\sigma \dot{y}^{\beta }}},
\nonumber\\
\chi _1^{\prime 11}
& =\left(
-z+\dot{y}^\gamma z_{\dot{y}^\gamma }
\right)
\tfrac{\partial }{\partial z_x}
+\dot{y}^\beta z_{\dot{y}^{\alpha }\dot{y}^\beta }
\tfrac{\partial }{\partial z_{x\dot{y}^\alpha }},
\label{chiprime111}\\
\bar{\chi }_{\alpha}^{\prime11}
& =-z_{\dot{y}^\alpha }
\tfrac{\partial }{\partial z_x}
-z_{\dot{y}^\beta \dot{y}^\alpha }
\tfrac{\partial }{\partial z_{x\dot{y}^\beta }},
\label{chibarraprime11alfa}\\
\tilde{\chi }_\alpha ^{\prime\beta 1}
& =-\dot{y}^\beta z_{\dot{y}^\alpha }
\tfrac{\partial }{\partial z_x}
-z_{\dot{y}^\alpha }
\tfrac{\partial }{\partial z_{y^\beta }}
-\dot{y}^\beta z_{\dot{y}^\sigma \dot{y}^\alpha }
\tfrac{\partial }{\partial z_{x\dot{y}^\sigma }}
-z_{\dot{y}^\alpha }
\tfrac{\partial }{\partial z_{x\dot{y}^\beta }}
\nonumber\\
& -z_{\dot{y}^\sigma \dot{y}^\alpha }
\tfrac{\partial }{\partial z_{y^\beta \dot{y}^\sigma }},
\nonumber\\
\hat{\chi }_\beta ^{\prime \alpha \gamma }
& =-\dot{y}^\gamma z_{\dot{y}^\beta }
\tfrac{\partial }{\partial z_{y^\alpha }}
-z_{\dot{y}^\beta }
\tfrac{\partial }{\partial z_{y^\alpha \dot{y}^\gamma }}
-\dot{y}^\gamma z_{\dot{y}^\sigma \dot{y}^\beta }
\tfrac{\partial }{\partial z_{y^\alpha \dot{y}^\sigma }}.
\nonumber
\end{align}
From \eqref{chibarraprime11alfa}, we have
$\tfrac{\partial }{\partial z_x }
=-\frac{1}{z_{\dot{y}^{\alpha_0}}}
\bar{\chi }_{\alpha_0}^{\prime11}
-\frac{z_{\dot{y}^\beta \dot{y}^{\alpha _0}}}
{z_{\dot{y}^{\alpha _0}}}
\tfrac{\partial }{\partial z_{x\dot{y}^\beta }}$
on the dense open subset $z_{\dot{y}^{\alpha _0}}\neq 0$.
Replacing the previous formula into \eqref{chiprime111}
and letting free the index $\alpha_0$, we obtain
\begin{equation}
C_\beta ^\alpha
\tfrac{\partial }{\partial z_{x\dot{y}^\beta }}
=v^\alpha ,\quad 1\leq \alpha \leq m,
\label{system}
\end{equation}
where $C_{\beta}^\alpha
=\left( z-\dot{y}^\gamma z_{\dot{y}^\gamma }\right)
z_{\dot{y}^\beta \dot{y}^\alpha }
+\dot{y}^\gamma z_{\dot{y}^\alpha }
z_{\dot{y}^\gamma \dot{y}^\beta }$,
$v^\alpha =\left( z-\dot{y}^\gamma z_{\dot{y}^\gamma }\right)
\bar{\chi }_\alpha ^{\prime 11}
-z_{\dot{y}^\alpha }\chi _1^{\prime 11}$.
On the dense open subset $O^{\prime 2}$ defined by
\[
0\neq \det
\left(
C_\beta ^\alpha
\right) _{\alpha ,\beta =1}^m
=z\left(
z-\dot{y}^\gamma z_{\dot{y}^\gamma }
\right) ^{m-1}\det
\left(
z_{\dot{y}^\alpha \dot{y}^\beta }
\right) _{\alpha ,\beta =1}^m,
\]
we can solve \eqref{system} for
$\tfrac{\partial }{\partial z_{x\dot{y}^{\beta }}}$,
thus proving that
$\tfrac{\partial }{\partial z_{x\dot{y}^\beta }},
\tfrac{\partial }{\partial z_x}\in
\left.
\mathcal{D}^2
\right\vert _{O^{\prime 2}}$.
Hence the distribution $\mathcal{D}^2$ is spanned by
$\tfrac{\partial }{\partial x}$,
$\tfrac{\partial }{\partial y^\alpha }$,
$\tfrac{\partial }{\partial z_x}$,
$\tfrac{\partial }{\partial z_{xx}}$,
$\tfrac{\partial }{\partial z_{xy^\alpha }}$,
$\tfrac{\partial }{\partial z_{x\dot{y}^\beta }}$,
$\tfrac{\partial }{\partial z_{y^\alpha y^\beta }}$,
$\alpha\leq\beta$, and the following vector fields:
\begin{align}
\chi _1^{\prime \prime 1}
& =-z\tfrac{\partial }{\partial z}-\dot{y}^\alpha
\tfrac{\partial }{\partial\dot{y}^\alpha }
-2z_x\tfrac{\partial }{\partial z_x}
-z_{y^\alpha }\tfrac{\partial }{\partial z_{y^\alpha }}
+\sum\nolimits_{\alpha \leq \beta }
z_{\dot{y}^\alpha \dot{y}^\beta }
\tfrac{\partial }
{\partial z_{\dot{y}^\alpha \dot{y}^\beta }},
\nonumber\\
\bar{\chi }_\alpha ^{\prime \prime 1}
& =\tfrac{\partial }{\partial \dot{y}^\alpha },
\label{chi_pp_1alpha}\\
\tilde{\chi }_\alpha ^{\prime \prime \beta}
& =\dot{y}^\beta \tfrac{\partial }{\partial\dot{y}^\alpha }
-z_{y^\alpha }\tfrac{\partial }{\partial z_{y^\beta }}
-z_{\dot{y}^\alpha }
\tfrac{\partial }{\partial z_{\dot{y}^\beta }}
\nonumber\\
& -z_{y^\sigma \dot{y}^\alpha }
\tfrac{\partial }{\partial z_{y^\sigma \dot{y}^\beta }}
-z_{y^\alpha \dot{y}^\sigma }
\tfrac{\partial }{\partial z_{y^\beta \dot{y}^\sigma }}
-(1+\delta_\sigma ^\beta )z_{\dot{y}^\alpha \dot{y}^\sigma }
\tfrac{\partial }{\partial z_{\dot{y}^\sigma \dot{y}^\beta }},
\nonumber\\
\tilde{\chi }_\alpha ^{\prime \prime \beta 1}
& =-z_{\dot{y}^\alpha }
\tfrac{\partial }{\partial z_{y^\beta }}
-z_{\dot{y}^\sigma \dot{y}^\alpha }
\tfrac{\partial }{\partial z_{y^\beta \dot{y}^\sigma }},
\label{chi_pp_beta1alpha}\\
\hat{\chi }_\beta ^{\prime \prime \alpha \gamma}
& =-\dot{y}^\gamma z_{\dot{y}^\beta }
\tfrac{\partial }{\partial z_{y^\alpha }}
-z_{\dot{y}^\beta }
\tfrac{\partial }{\partial z_{y^\alpha \dot{y}^\gamma }}
-\dot{y}^\gamma z_{\dot{y}^\sigma \dot{y}^\beta }
\tfrac{\partial }{\partial z_{y^\alpha \dot{y}^\sigma }}.
\label{chi_pp_alphagammabeta}
\end{align}
From \eqref{chi_pp_1alpha} we have
$\tfrac{\partial }{\partial \dot{y}^\alpha }
=\bar{\chi }_\alpha^{\prime \prime 1}$,
and on the dense open subset
$z_{\dot{y}^{\alpha _0}}\neq 0$, from
\eqref{chi_pp_beta1alpha} we have
$\tfrac{\partial }{\partial z_{y^\beta }}
=-\tfrac{1}{z_{\dot{y}^{\alpha _0}}}
\tilde{\chi }_{\alpha _0}^{\prime \prime \beta 1}
-\tfrac{z_{\dot{y}^\sigma \dot{y}^{\alpha _0}}}
{z_{\dot{y}^{\alpha _0}}}
\tfrac{\partial }{\partial z_{y^\beta \dot{y}^\sigma }}$.
Replacing the previous formula into
\eqref{chi_pp_alphagammabeta}, we obtain
\begin{equation}
\bar{C}_{\beta \sigma }^\gamma
\tfrac{\partial }{\partial z_{y^\alpha \dot{y}^\sigma }}
=\bar{v}_{\beta}^{\alpha\gamma},
\quad
\alpha, \beta, \gamma =1,\dotsc,m,
\label{system2}
\end{equation}
with $\bar{C}_{\beta \sigma }^\gamma
=\dot{y}^\gamma z_{\dot{y}^\beta }
\tfrac{z_{\dot{y}^\sigma \dot{y}^{\alpha _0}}}
{z_{\dot{y}^{\alpha _0}}}
-z_{\dot{y}^\beta }
\delta _\sigma ^\gamma
-\dot{y}^\gamma z_{\dot{y}^\sigma \dot{y}^\beta }$,
$\bar{v}_\beta ^{\alpha \gamma }
=\hat{\chi }_\beta ^{\prime \prime \alpha \gamma }
-\tfrac{\dot{y}^\gamma z_{\dot{y}^\beta }}
{z_{\dot{y}^{\alpha _0}}}
\tilde{\chi }_{\alpha _0}^{\prime \prime \alpha 1}$.
Letting $\alpha =\alpha _1$, $\beta =\beta _0$,
the system \eqref{system2} transforms
into the following system of $m$ equations
and $m$ unknowns:
$\bar{C}_{\beta _0\sigma }^\gamma
\tfrac{\partial }{\partial z_{y^{\alpha _1}
\dot{y}^\sigma }}=\bar{v}_{\beta _0}^{\alpha _1\gamma }$,
and, in particular, for $\beta _0=\alpha _0$, we have
\[
\bar{C}_{\alpha _0\sigma }^\gamma
=\dot{y}^\gamma z_{\dot{y}^{\alpha _0}}
\tfrac{z_{\dot{y}^\sigma \dot{y}^{\alpha _0}}}
{z_{\dot{y}^{\alpha _0}}}
-z_{\dot{y}^{\alpha _0}}
\delta _\sigma ^\gamma -\dot{y}^\gamma
z_{\dot{y}^\sigma \dot{y}^{\alpha_0}}
=-z_{\dot{y}^{\alpha _0}}\delta _\sigma ^\gamma ,
\]
\[
\det(\bar{C}_{\alpha _0\sigma }^\gamma )_{\gamma ,\sigma =1}^m
=\det
\left(
-z_{\dot{y}^{\alpha _0}}\delta _\sigma ^\gamma
\right) _{\gamma ,\sigma =1}^m
=(-z_{\dot{y}^{\alpha _0}})^m\neq 0.
\]
Therefore, we can solve the equations \eqref{system2}
with respect to
$\tfrac{\partial }{\partial z_{y^\alpha \dot{y}^\sigma }}$,
thus concluding that the distribution $\mathcal{D}^2$ is
spanned by $\tfrac{\partial }{\partial x}$,
$\tfrac{\partial }{\partial y^\alpha }$,
$\tfrac{\partial }{\partial\dot{y}^\alpha }$,
$\tfrac{\partial }{\partial z_x}$,
$\tfrac{\partial }{\partial z_{y^\alpha }}$,
$\tfrac{\partial }{\partial z_{xx}}$,
$\tfrac{\partial }{\partial z_{xy^\alpha }}$,
$\tfrac{\partial }{\partial z_{x\dot{y}^\alpha }}$,
$\tfrac{\partial }{\partial z_{y^{\alpha }y^\beta }}$,
$\alpha\leq\beta$,
$\tfrac{\partial }{\partial z_{y^\alpha \dot{y}^\sigma }}$,
and the following $m^2+1$ additional vector fields:
\begin{align*}
\zeta _1^1
& =-z\tfrac{\partial }{\partial z}
+\sum\nolimits_{\alpha \leq \beta }
z_{\dot{y}^\alpha \dot{y}^\beta }
\tfrac{\partial }{\partial z_{\dot{y}^\alpha \dot{y}^\beta }},\\
\bar{\zeta}_{\alpha}^\beta
& =-z_{\dot{y}^\alpha }
\tfrac{\partial }{\partial z_{\dot{y}^\beta }}
-(1+\delta _\sigma ^\beta )
z_{\dot{y}^{\alpha }\dot{y}^\sigma }
\tfrac{\partial }{\partial z_{\dot{y}^\sigma \dot{y}^\beta }}.
\end{align*}
For every $j_{j_x^1\sigma }^2\mathcal{L}\in O^{\prime 2}$
we thus have
\begin{align*}
\dim \left.
\mathcal{D}^2\right\vert _{j_{j_x^1\sigma }^2\mathcal{L}}
& =\tfrac{3}{2}m^2+\tfrac{11}{2}m+3\\
& +\operatorname{rank}
\left\{
\left.
\zeta_1^1
\right\vert _{j_{j_x^1\sigma }^2\mathcal{L}},
\left.  \bar{\zeta }_\alpha^\beta
\right\vert _{j_{j_x^1\sigma }^2\mathcal{L}}
:\alpha,\beta=1,\dotsc,m
\right\} .
\end{align*}
Therefore, we only need to prove that the rank of the system
$\zeta _1^1|_{j_{j_x^1\sigma }^2\mathcal{L}},
\bar{\zeta }_\alpha^\beta |_{j_{j_x^1\sigma }^2\mathcal{L}}$,
$\alpha ,\beta =1,\dotsc,m$, is $\frac{1}{2}m\left( m+3\right) $.
To do this, we choose coordinates
$(y^\alpha )_{\alpha =1}^m$ adapted to the Hessian metric
$\operatorname{Hess}\nolimits_{j_{x_0}^1\! \sigma _0}(\mathcal{L})$,
namely
\[
\begin{array}
[c]{rl}
z_{\dot{y}^\alpha \dot{y}^\beta }(j_{j_x^1\sigma }^2(\mathcal{L}))
=\varepsilon _\alpha \delta _{\alpha \beta},
& \alpha ,\beta =1,\dotsc,m,\\
\varepsilon _\alpha = & \left\{
\begin{array}
[c]{ll}
+1, & 1\leq \alpha \leq m^+,\\
-1, & 1+m^+\leq \alpha \leq m,
\end{array}
\right.
\end{array}
\]
the pair $(m^+,m^-)$, $m^-+m^+=m$, being the signature of
$\operatorname{Hess}\nolimits_{j_{x_0}^1\! \sigma _0}(\mathcal{L})$.
Hence
\[
\begin{array}
[c]{ll}
\left.
\zeta _1^1\right\vert _{j_{j_x^1\sigma }^2\mathcal{L}}
=\left.
-z\tfrac{\partial }{\partial z}
+\varepsilon _\alpha
\tfrac{\partial }{\partial z_{\dot{y}^\alpha \dot{y}^\alpha }}
\right\vert _{j_{j_x^1\sigma }^2\mathcal{L}},
\smallskip & \\
\left.
\bar{\zeta}_{\alpha}^\alpha
\right\vert _{j_{j_x^1\sigma }^2\mathcal{L}}
=\left.
-z_{\dot{y}^\alpha }
\tfrac{\partial }{\partial z_{\dot{y}^\alpha }}
-2\varepsilon _\alpha
\tfrac{\partial }{\partial z_{\dot{y}^\alpha \dot{y}^\alpha }}
\right\vert _{j_{j_x^1\sigma }^2\mathcal{L}},
\smallskip & \\
\left.
\bar{\zeta }_\alpha ^\beta
\right\vert _{j_{j_x^1\sigma }^2\mathcal{L}}
=\left.
-z_{\dot{y}^\alpha }
\tfrac{\partial }{\partial z_{\dot{y}^\beta }}
-\varepsilon _\alpha
\tfrac{\partial }{\partial z_{\dot{y}^\alpha \dot{y}^\beta }}
\right\vert _{j_{j_x^1\sigma }^2\mathcal{L}},
& \alpha \neq \beta.
\end{array}
\]
We split the third group above and choose new generators as follows:
\[
\begin{array}
[c]{ll}
\left.
\bar{\zeta}_{\alpha}^\beta
\right\vert _{j_{j_x^1\sigma }^2\mathcal{L}}
=\left. -z_{\dot{y}^\alpha }
\tfrac{\partial }{\partial z_{\dot{y}^\beta }}
-\varepsilon _\alpha
\tfrac{\partial }{\partial z_{\dot{y}^\alpha \dot{y}^\beta }}
\right\vert _{j_{j_x^1\sigma }^2\mathcal{L}},
& \alpha <\beta ,
\smallskip\\
\left.
\hat{\zeta}_{\alpha }^\beta
\right\vert _{j_{j_x^1\sigma }^2\mathcal{L}}
=\varepsilon_{\beta }\bar{\zeta }_{\alpha}^\beta
-\varepsilon_{\alpha }\bar{\zeta }_\beta^\alpha
=\varepsilon_{\alpha}
z_{\dot{y}^\beta }
\tfrac{\partial }{\partial z_{\dot{y}^\alpha }}
-\varepsilon _{\beta}z_{\dot{y}^\alpha }
\tfrac{\partial }{\partial z_{\dot{y}^\beta }},
& \beta <\alpha .
\end{array}
\]
The system $\zeta _1^1|_{j_{j_x^1\sigma }^2\mathcal{L}}$,
$\bar{\zeta }_{\alpha}^\alpha |_{j_{j_x^1\sigma }^2\mathcal{L}}$,
$\hat{\zeta }_\alpha ^1|_{j_{j_x^1\sigma }^2\mathcal{L}}$,
$2\leq \alpha \leq m$,
$\bar{\zeta}_\alpha ^\beta |_{j_{j_x^1\sigma }^2\mathcal{L}}$,
$1\leq \alpha<\beta\leq m$, is readily seen to be independent.
As for the rest of the vectors, one obtains the relations,
$\hat{\zeta}_\alpha ^\beta |_{j_{j_x^1\sigma }^2\mathcal{L}}
=\frac{z_{\dot{y}^\beta }}{z_{\dot{y}^1}}
\hat{\zeta}_{\alpha}^1|_{j_{j_x^1\sigma }^2\mathcal{L}}
-\frac{z_{\dot{y}^\alpha }}{z_{\dot{y}^1}}
\hat{\zeta}_\beta ^1|_{j_{j_x^1\sigma }^2\mathcal{L}}$,
$2\leq\beta<\alpha\leq m$, so that the rank of the system
$\zeta_1^1|_{j_{j_x^1\sigma }
^2\mathcal{L}}$, $\bar{\zeta}_{\alpha}^\beta |_{j_{j_x^1\sigma }
^2\mathcal{L}}$, $\alpha,\beta=1,\dotsc,m$, is
$\frac{1}{2}m\left( m+3\right) $. Therefore
\[
\dim\left. \mathcal{D}^2\right\vert _{j_{j_x^1\sigma }^2\mathcal{L}}
=\tfrac{3}{2}m^2+\tfrac{11}{2}m+3+\tfrac{1}{2}m\left( m+3\right)
=2m^2+7m+3.
\]
\end{proof}
\section{Metric invariants}
The Hessian metric $\operatorname{Hess}(\mathcal{L})$ of a function
$\mathcal{L}\in C^\infty (J^1(\mathbb{R},M))$ is the section of the bundle
$\varkappa\colon S^2\left[ V^\ast (p^{10})\right] \to J^1(\mathbb{R},M)$
defined in the formula \eqref{Hessian_metric} of the subsection \ref{Hessian}.

Let $\mathcal{M}\subset S^2\left[  V^\ast (p^{10})\right]  $ be the open
subbundle whose fibre $\mathcal{M}_{j_x^1\sigma }$ over $j_x^1\sigma $
is the set of non-degenerate symmetric bilinear forms
$V_{j_x^1\sigma }(p^{10})\times V_{j_x^1\sigma }(p^{10})\to \mathbb{R}$.

Every $\Phi \in \operatorname{Aut}(p)$ induces an isomorphism
$(\Phi ^{(1)})_\ast \colon V_{j_x^1\sigma }(p^{10})
\to  V_{j_x^1 \sigma }(p^{10})$,
as $p^{10}\circ \Phi ^{(1)}=\Phi \circ p^{10}$.
Hence $\Phi $ induces a diffeomorphism
$\Phi _{\mathcal{M}}^{(1)}\colon \mathcal{M}\to \mathcal{M}$
defined by
$\Phi _{\mathcal{M}}^{(1)}(g_{j_x^1\sigma })
=(\Phi ^{(1)})^{-1~\ast}(g_{j_x^1\sigma })$ so that
$\varkappa\circ \Phi _{\mathcal{M}}^{(1)}=\Phi \circ \varkappa $.

A function $I\in C^\infty (J^r\mathcal{M})$ is said to be a metric
invariant of order $r$ (cf.\ \cite[\S 2]{JA}) if the following
property holds:
\[
\begin{array}
[c]{rlll}
I(\Psi _{\mathcal{M}}^{(r)}(j_{j_x^1\sigma }^rg))
& =I(j_{j_x^1\sigma }^rg),
& \forall \Psi \in \operatorname{Diff}J^1(\mathbb{R},M),
& \\
& & \forall  j_x^1\sigma \in J^1(\mathbb{R},M),
& \forall  g\in \Gamma( \varkappa ),
\end{array}
\]
where $\Psi _{\mathcal{M}}\colon \mathcal{M\to M}$ is defined
in the same way as $\Phi _{\mathcal{M}}^{(1)}$, and
$\Psi_{\mathcal{M}}^{(r)}$ is its $r$-th jet prolongation.

In particular, if $I$ is a metric invariant, then
$I((\Phi _{\mathcal{M}}^{(1)})^{(r)}(j_{j_x^1\sigma }^rg))
=I(j_{j_x^1\sigma }^rg)$, $\forall \Phi \in \operatorname{Aut}(p)$.

\begin{proposition}
Every metric invariant of order $r$ induces an invariant in the sense
of \emph{Proposition \ref{proposition1} of order }$r+2$. Therefore,
if $\dim M=m\geq 2$, then there exists
$\operatorname{Aut}(p)$-invariant functions that cannot be obtained
as derivatives of the second-order basic invariant $I$
defined in \emph{Proposition \ref{proposition_V_I}}, but all of these
invariants are of order $\geq 3$.
\end{proposition}
\begin{proof}
Let $O^{\prime 2}\subset J^2(q)$ be the dense open subset of elements
$j_{j_{x_0}^1\sigma _0}^2(\mathcal{L})$ for which the Hessian metric
$\operatorname{Hess}_{j_{x_0}^1\sigma _0}(\mathcal{L})$ is non-singular
and $\mathcal{L}(j_{x_0}^1\sigma _0)\neq 0$, as in
\S \ref{basic_invariant}.

Let $q^{kh}\colon J^k(q)\to J^h(q)$, $k\geq h$,
be the canonical projection. For every $r\geq 2$,
let $O^{\prime r}$ be the dense open subset in
$J^r(q)$ given by
$O^{\prime r}=(q^{r2})^{-1}(O^{\prime2})$ and for every
$r\geq 0$ let $\Theta ^r\colon J^{r+2}(O^{\prime2})
\to J^r(\mathcal{M})$ the fibred map defined as
\[
\Theta ^r\left( j_{j_x^1\sigma }^{r+2}\mathcal{L}\right)
=j_{j_x^1\sigma }^r\left( \operatorname{Hess}(\mathcal{L})\right) ,
\]
which is $\operatorname{Aut}(p)$-equivariant with respect to the natural
actions on these spaces by virtue of Proposition \ref{covv}. Hence, every
invariant function $I\in C^\infty (J^r\mathcal{M})$ induces an invariant
function $I\circ\Theta ^r$ of order $r+2$.

Moreover, as is well known, the basic metric invariants are the scalar
contractions of the successive covariant differentials of the curvature
tensor of the corresponding Levi-Civita of a metric. Hence, for a general
metric $g\in \Gamma (\varkappa )$, every metric invariant is of order
$\geq 2 $, but the curvature tensor of a Hessian metric
\[
\begin{array}
[c]{rlrl}
g= & \sum _{h,i=1}^{m}g_{hi}d\dot{y}^{h}\otimes d\dot{y}^{i},
& g_{hi}=g_{ih} &
=\frac{\partial^2\mathcal{L}}{\partial\dot{y}^{h}\partial\dot{y}^i},
\end{array}
\]
depends on the third derivatives of $\mathcal{L}$ (e.g., see
\cite[Proposition 2.3--(1)]{Shima}), thus concluding.
\end{proof}
\begin{example}
If $\dim M=m=2$, then the basic metric invariant is the Gaussian curvature
(cf.\ \cite[formula (1.9)]{Duistermaat}):
\begin{equation*}
\begin{array}{l}
4\left[ \frac{\partial ^2\mathcal{L}}
{\partial \dot{y}^1\partial \dot{y}^1}
\frac{\partial ^2\mathcal{L}}
{\partial \dot{y}^2\partial \dot{y}^2}
-\left( \frac{\partial ^2\mathcal{L}}
{\partial \dot{y}^{1}\partial \dot{y}^2}
\right) ^2\right] ^{2}K= \\
-\frac{\partial ^{2}\mathcal{L}}
{\partial \dot{y}^2\partial \dot{y}^2}
\left[ \frac{\partial ^3\mathcal{L}}
{\partial \dot{y}^1\partial \dot{y}^1\partial \dot{y}^1}
\frac{\partial ^3\mathcal{L}}
{\partial \dot{y}^1\partial \dot{y}^2\partial \dot{y}^2}
-\left( \frac{\partial ^3\mathcal{L}}
{\partial \dot{y}^1\partial \dot{y}^1\partial \dot{y}^2}
\right) ^2\right] \\
+\frac{\partial ^{2}\mathcal{L}}
{\partial \dot{y}^1\partial \dot{y}^2}
\left[ \frac{\partial ^{3}\mathcal{L}}
{\partial \dot{y}^1\partial \dot{y}^1\partial \dot{y}^1}
\frac{\partial ^3\mathcal{L}}
{\partial \dot{y}^2\partial \dot{y}^2\partial \dot{y}^2}
-\frac{\partial ^3\mathcal{L}}
{\partial \dot{y}^1\partial \dot{y}^1\partial \dot{y}^2}
\frac{\partial ^3\mathcal{L}}
{\partial \dot{y}^1\partial \dot{y}^2\partial \dot{y}^2}
\right] \\
-\frac{\partial ^2\mathcal{L}}
{\partial \dot{y}^1\partial \dot{y}^1}
\left[ \frac{\partial ^3\mathcal{L}}
{\partial \dot{y}^1\partial \dot{y}^1\partial \dot{y}^2}
\frac{\partial ^3\mathcal{L}}
{\partial \dot{y}^2\partial \dot{y}^2\partial \dot{y}^2}
-\left( \frac{\partial ^3\mathcal{L}}
{\partial \dot{y}^1\partial \dot{y}^2\partial \dot{y}^2}
\right) ^2\right] ,
\end{array}
\end{equation*}
which is an invariant of third order.
\end{example}

\bigskip

\noindent (M.C.L) \textsc{Instituto de Ciencias Matem\'aticas,
CSIC-UAM-UC3M-UCM, Departamento de Geometr\'{\i}a y Topolog\'{\i}a, Facultad
de Matem\'{a}ticas, UCM, Avda. Complutense s/n, 28040-Madrid, Spain}

\noindent \emph{E-mail:\/} \texttt{mcastri@mat.ucm.es}

\medskip

\noindent (J.M.M) \textsc{Instituto de Tecnolog\'{\i}as F\'{\i}sicas y de la
Informaci\'on, Consejo Superior de Investigaciones Cien\-t\'{\i}\-fi\-cas,
C/ Serrano 144, 28006-Madrid, Spain}

\noindent \textit{E-mail:\/} \texttt{jaime@iec.csic.es}

\medskip

\noindent (E.R.M) \textsc{Departamento de Matem\'atica Aplicada, Escuela
T\'ecnica Superior de Arquitectura, Universidad Polit\'ecnica de Madrid,
Avda.\ Juan de Herrera 4,\ 28040-Madrid, Spain}

\noindent \emph{E-mail:\/} \texttt{eugenia.rosado@upm.es}

\end{document}